\documentclass[11pt,oneside]{article}
\usepackage{graphicx} 

\usepackage[utf8]{inputenc}
\usepackage[english]{babel}
\usepackage[nottoc]{tocbibind}

\usepackage{lmodern}

\usepackage{mathtools}
\usepackage{fancyhdr}
\usepackage{amsfonts}
\usepackage{amsmath}
\usepackage{amssymb}
\usepackage{mathrsfs}
\usepackage{cite}
\usepackage{enumitem}
\usepackage{import}
\usepackage{tabularx}
\usepackage{latexsym}
\usepackage{shortvrb}
\usepackage{graphicx}
\usepackage{makeidx}
\usepackage[all]{xy}
\usepackage{amsthm}
\usepackage{fdsymbol}
\usepackage{physics}
\usepackage{imakeidx}

\usepackage{hyperref}
\hypersetup{
	colorlinks=true,
	linkcolor=blue,
	filecolor=magenta,      
	urlcolor=cyan,
}
\usepackage{geometry}

\usepackage{pgf,tikz,pgfplots}
\pgfplotsset{compat=1.17}
\usetikzlibrary{arrows}
\usetikzlibrary{matrix}

\newtheorem{thm}{Theorem}[section]

\newtheorem{prop}[thm]{Proposition}
\newtheorem{ex}[thm]{Example}
\newtheorem{cor}[thm]{Corollary}

\theoremstyle{definition}
\newtheorem{definition}[thm]{Definition}
\newtheorem{obs}[thm]{Observation}

\newtheorem{prob}[thm]{Problem}

\DeclareMathOperator{\Rel}{\mathbb{R}el}
\DeclareMathOperator{\Set}{\mathbf{Set}}

\DeclareMathOperator{\Span}{\mathbb{S}pan}
\DeclareMathOperator{\Bord}{\mathbb{B}ord}

\DeclareMathOperator{\Mod}{\mathbb{M}od}

\DeclareMathOperator{\Eq}{\mathbb{E}q}
\DeclareMathOperator{\Prof}{\mathbb{P}rof}

\usepackage{amsmath,amssymb}
\usepackage{comment}

\newcommand{\commMon}{\textbf{commMon}}
\newcommand{\stB}{B^*}

\newcommand{\decB}{(\stB,B)}

\newcommand{\dcat}{\textbf{dCat}}

\newcommand{\crossb}{B\rtimes_\Phi B^\ast}

\newcommand{\indFunct}{\Phi:\stB\to\commMon}

\makeindex

\title{Length of fully faithful framed bicategories}
\author{Juan Orendain}

\begin{document}

\maketitle

\begin{abstract}
\noindent The length of a double category is a numerical invariant measuring the 'work' it takes to reconstruct the double category from its globular data. The smallest possible length of a double category is 1. It is conjectured that framed bicategories are of length 1. In this paper we prove this conjecture for a particular class of framed bicategories, namely for those double categories for which all their unit squares are cartesian/opcartesian. We call these framed bicategories fully faithful/absolutely dense. 
\end{abstract}

\section{Introduction}

\noindent In the series of papers \cite{Orendain1,Orendain2,Orendain3,OrendainMaldonado}, the problem of lifting a bicategory $B$ to a double category $C$ along a given category of vertical arrows is considered. Constructing a square in a double category $C$, from the data of the horizontal bicategory $HC$ of $C$ and the category of vertical arrows $C_0$ of $C$, involves sequences of 2-dimensional compositions in different directions. These compositions are naturally organized into a double filtration of the category of squares of $C$. The length $\ell \varphi$ of a square $\varphi$ in the double category $C$, is a number in $\mathbb{N}$, measuring how long in this filtration one has to go to construct the square, and we thus consider it as a measure of how much 'work' it takes to construct the square $\varphi$ from globular data in $C$. The length $\ell C$ of a double category $C$ is the supremum of lengths of squares in $C$. We are concerned with the problem of computing the length of framed bicategories.

Bicategories are categories enriched over \textbf{Cat}, while double categories are categories internal to \textbf{Cat}. We thus think of the operation of lifting a bicategory $B$ to a double category $C$ along a category of vertical arrows $B^\ast$, as a way of internalizing $B$ along the category $B^\ast$. There are at least two situations in modern low dimensional category theory in which one wishes to internalize a bictegory $B$: When defining a symmetric monoidal structure on $B$, where the coherence data of the symmetric monoidal structure is more naturally expressed in terms of a given category of vertical arrows, see \cite{ShulmanFramed,ShulmanDerived}, think of the symmetric monoidal double category $\Mod$; and when equipping the category $B^\ast$ with proarrows, as in \cite{Wood1,Wood2}, think of the equipment $\Prof$. In both cases one is really interested in framed bicategories. There are more situations where this is desirable, e.g. \cite{KellyStreet} where internalizing provides a satisfactory expression of the mates correspondence, and \cite{BrownBook} where inernalizing provides a natural way of expressing and proving the HHSvK theorem. In all of these cases, the bicategory $B$, the category of vertical arrows $B^\ast$, and the double category $C$ are known. There are cases where $B$ and $B^\ast$ are known, but it is not clear that $C$ exists, see \cite{Bartels1,Orendain1} for the case of Hilbert bimodules, and normal morphisms of von Neumann algebras. There are other cases where more than one possible natural internalization of a pair $\decB$ can be constructed, see \cite[Section 3]{Orendain3} and \cite[Example 6.1]{OrendainMaldonado}. It is desirable to be able to decide if given a bicategory $B$ and a vertical category $B^\ast$, if the par $\decB$ can be internalized into a framed bicategory, if that is the case, to obtain information on what squares on a framed bicategory internalizing $\decB$ looks like, and ultimately, we would like to decide if there is a 'best' possible framed bicategory internalizing $\decB$. Knowing the length of a given framed bicategory is a step in this direction. The framed bicategories $\Mod,\Rel,\Prof,\Bord$ are known to be of length 1, and thus it is natural to conjecture that all framed bicategories are of length 1. In this paper we prove this conjecture for the case of framed bicategories for which every unit square is cartesian/opcartesian. We call these framed bicategories fully faithul/absolutely dense.  We moreover, provide an explicit description of the 'shape' of squares of both fully faithful and absolutely dense framed bicategories through an evaluation double functor from double categories of the form $\crossb$ defined in \cite{OrendainMaldonado}. Although our main motivation for this work is to compute the length of fully faithful/absolutely dense framed bicategories, we present our results in a more general setting. Our description of squares in fully faithful and absolutely dense framed bicategories and the computation of their lengths is done for double categories inducing $\pi_2$-indexings, which we introduce in Section \ref{s:inducingpi2index}. 

We briefly outline our strategy for proving our results. In \cite[Lemma 4.6]{Orendain1} it is proven that squares of length 1 in a double category admit a vertical subdivision of a specific form, the simplest of which is studied in \cite{OrendainMaldonado} and is called canonical form, see subsection \ref{ss:piindexings}. In \cite[Lemma 3.9]{OrendainMaldonado} it is proven that if every length 1 square in a double category $C$ is of canonical form, then $\ell C=1$. In order to prove that fully faithful/absolutely dense framed bicategories are of length 1 we thus prove that every length 1 square of a fully faithful/absolutely dense framed bicategory is of canonical form. To do this we leverage the notion of $\pi_2$-opindexing/$\pi_2$-indexing, defined in \cite{OrendainMaldonado}. A $\pi_2$-opindexing is a type of indexing, defined on a pair $\decB$ as above, implementing an abstract sliding operation of 2-cells in $B$ along morphisms in $B^\ast$. To every $\pi_2$-opindexing on $\decB$ a double category $\crossb$ is constructed, such that every square in $\crossb$ is of canonical form. In order to prove that every square in a fully faithful/absolutely dense framed bicategory $C$ is of canonical form, we first extract from $C$ a $\pi_2$-opindexing/$\pi_2$-indexing $\Phi$ on $H^\ast C=(C_0,HC)$, see Lemma \ref{lem:main}. In Theorem \ref{thm:initial} we prove that any double category for which a $\pi_2$-opindexing can be extracted in this way, and thus in particular any fully faithful/absolutely dense framed bicategory can be parametrized, in some sense, by a double category of the form $\crossb$, and thus all its length 1 squares are canonical and the double category is of length 1. In more categorical terms, we prove that double categories of the form $\crossb$ are initial in a category of double categories inducing the $\pi_2$-opindexing/$\pi_2$-indexing $\Phi$. The results in this note rely heavily on results appearing in the series \cite{Orendain1,Orendain2,Orendain3,OrendainMaldonado}. We will develop a pictorial language that we believe will make our arguments more transparent.

We now present the contents of this paper. In Section \ref{s:Prel} we briefly review the notions of internalization, globularly generated double category, length, $\pi_2$-indexing, and framed bicategory. In Section \ref{s:inducingpi2index} we introduce double categories inducing $\pi_2$-indexing/$\pi_2$-opindexings, and we prove that double categories of the form $\crossb$ are initial in a category of double categories inducing a $\pi_2$-indexing/$\pi_2$-opindexing $\Phi$, and that in particular, double categories inducing $\pi_2$-indexings/$\pi_2$-opindexings are of length 1. In Section \ref{s:fully faithful} we introduce fully faithful and absolutely dense framed bicategories, and present relevant examples. We prove that every framed bicategory contains a fully faithful/absolutely dense framed bicategory universally. In Section \ref{s:regframedidexngs} we present the main results of this note. We prove that every fully faithful/absolutely dense framed bicategory induces a $\pi_2$-opindexing/$\pi_2$-indexing. As a corollary of this we will prove that every fully faithful/absolutely dense framed bicategory is of length 1. We present examples of $\pi_2$-opindexings induced by fully faithful framed bicategories.

\

\noindent \textbf{Acknowledgements:} The author would like to thank Bob Par\'e, Nathanael Arkor, and Rub\'en Maldonado-Herrera for their helpful comments on the paper.
 
\section{Preliminaries}\label{s:Prel}

\subsection{Double categories}\label{ss:Doubelcats}

\noindent We will write $C_0,C_1$ for the category of objects and the category of morphisms of a double category $C$; $L,R$ for the left and right frame functors of $C$, and $U$ for the unit (horizontal identity) functor of $C$. Following \cite{ShulmanDerived} we will write $\boxminus,\boxvert$ for the vertical and horizontal composition operations of $C$ respectively. We will represent squares in double categories as diagrams of the form:
\begin{center}
    
\tikzset{every picture/.style={line width=0.75pt}} 

\begin{tikzpicture}[x=0.75pt,y=0.75pt,yscale=-1,xscale=1]

\draw   (240.73,120.3) -- (300.9,120.3) -- (300.9,180.47) -- (240.73,180.47) -- cycle ;

\draw (263.92,144.91) node [anchor=north west][inner sep=0.75pt]  [font=\scriptsize]  {$\varphi $};
\draw (225.07,145.2) node [anchor=north west][inner sep=0.75pt]  [font=\scriptsize]  {$f$};
\draw (306.64,144.91) node [anchor=north west][inner sep=0.75pt]  [font=\scriptsize]  {$g$};
\draw (265,105) node [anchor=north west][inner sep=0.75pt]  [font=\scriptsize]  {$\alpha $};
\draw (265.92,185.26) node [anchor=north west][inner sep=0.75pt]  [font=\scriptsize]  {$\beta $};
\draw (225.07,113.2) node [anchor=north west][inner sep=0.75pt]  [font=\scriptsize]  {$a$};
\draw (225.07,174.2) node [anchor=north west][inner sep=0.75pt]  [font=\scriptsize]  {$c$};
\draw (305.07,174.2) node [anchor=north west][inner sep=0.75pt]  [font=\scriptsize]  {$d$};
\draw (306.07,113.2) node [anchor=north west][inner sep=0.75pt]  [font=\scriptsize]  {$b$};

\end{tikzpicture}

\end{center}

\noindent where squares are read from left to right and from top to bottom, i.e. the edges $f,g$ are the left and right frame of the square $\varphi$, while the edges $\alpha,\beta$ are the vertical domain and codomain of $\varphi$ respectively. We represent (horizontal or vertical) identities as red edges. We represent horizontal identity squares as squares marked with the letter $U$:
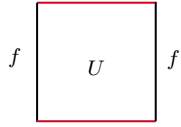
\begin{center}

\tikzset{every picture/.style={line width=0.75pt}} 

\begin{tikzpicture}[x=0.75pt,y=0.75pt,yscale=-1,xscale=1]

\draw [color={rgb, 255:red, 208; green, 2; blue, 27 }  ,draw opacity=1 ]   (300.07,139.3) -- (360.23,139.3) ;
\draw    (300.07,139.3) -- (300.07,199.47) ;
\draw [color={rgb, 255:red, 208; green, 2; blue, 27 }  ,draw opacity=1 ]   (299.73,199.13) -- (359.9,199.13) ;
\draw [color={rgb, 255:red, 0; green, 0; blue, 0 }  ,draw opacity=1 ]   (359.9,138.97) -- (359.9,199.13) ;

\draw (284,161.07) node [anchor=north west][inner sep=0.75pt]  [font=\scriptsize]  {$f$};
\draw (364,161.73) node [anchor=north west][inner sep=0.75pt]  [font=\scriptsize]  {$f$};
\draw (324,166.83) node [anchor=north west][inner sep=0.75pt]  [font=\scriptsize]  {$U$};

\end{tikzpicture}

\end{center}
\noindent An important class of squares is the collection squares of the form
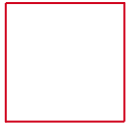
\begin{center}

\tikzset{every picture/.style={line width=0.75pt}} 

\begin{tikzpicture}[x=0.75pt,y=0.75pt,yscale=-1,xscale=1]

\draw [color={rgb, 255:red, 208; green, 2; blue, 27 }  ,draw opacity=1 ]   (320.07,159.3) -- (380.23,159.3) ;
\draw [color={rgb, 255:red, 208; green, 2; blue, 27 }  ,draw opacity=1 ]   (320.07,159.3) -- (320.07,219.47) ;
\draw [color={rgb, 255:red, 208; green, 2; blue, 27 }  ,draw opacity=1 ]   (319.73,219.13) -- (379.9,219.13) ;
\draw [color={rgb, 255:red, 208; green, 2; blue, 27 }  ,draw opacity=1 ]   (379.9,158.97) -- (379.9,219.13) ;

\end{tikzpicture}

\end{center}
\noindent We will write $\pi_2(C,a)$ for the set of all such squares with vertex $a$. By the Eckmann-Hilton argument, the horizontal composition and the vertical composition of such squares coincide and are commutative. The set $\pi_2(C,a)$ is thus a commutative monoid. Globular squares are represented by diagrams of the form
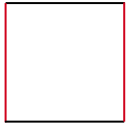
\begin{center}

\tikzset{every picture/.style={line width=0.75pt}} 

\begin{tikzpicture}[x=0.75pt,y=0.75pt,yscale=-1,xscale=1]

\draw [color={rgb, 255:red, 0; green, 0; blue, 0 }  ,draw opacity=1 ]   (300.07,119.3) -- (360.23,119.3) ;
\draw [color={rgb, 255:red, 208; green, 2; blue, 27 }  ,draw opacity=1 ]   (300.07,119.3) -- (300.07,179.47) ;
\draw [color={rgb, 255:red, 0; green, 0; blue, 0 }  ,draw opacity=1 ]   (299.73,179.13) -- (359.9,179.13) ;
\draw [color={rgb, 255:red, 208; green, 2; blue, 27 }  ,draw opacity=1 ]   (359.9,118.97) -- (359.9,179.13) ;

\end{tikzpicture}

\end{center}
\noindent The collection of globular squares of a double category $C$ forms a bicategory $HC$, the horizontal bicategory of $C$. Bicategories are precisely the double categories whose squares are all globular.

\subsection{Globularly generated double categories}\label{ss:ggdoublecats}
\noindent A decoration of a bicategory $B$ is a category $B^\ast$ having the same collection of objects as $B$. If $B^\ast$ is a decoration of $B$ we say that $\decB$ is a \textbf{decorated bicategory}. The pair $(C_0,HC)$ is a decorated bicategory for every double category $C$. We write $H^*C$ for this decorated bicategory, and we call it the \textbf{decorated horizontalization} of $C$. We are interested in the following problem:

\begin{prob}\label{prob:Internal}
Let $\decB$ be a decorated bicategory. Find double categories $C$ satisfying the equation $H^*C=\decB$.
\end{prob}

\noindent We understand problem \ref{prob:Internal} as the problem of lifting a bicategory $B$ to a double category through an orthogonal direction, provided by $\stB$. A solution to Problem \ref{prob:Internal} for a decorated bicategory $\decB$ will be called an \textbf{internalization} of $B$. Globularly generated double categories were introduced in \cite{Orendain1} as minimal solutions to Problem \ref{prob:Internal}. A double category $C$ is \textbf{globularly generated} if $C$ is generated, as a double category, by its collection of globular squares. Pictorially a double category $C$ is globularly generated if every square of $C$ can be written as vertical and horizontal compositions of squares of the form:
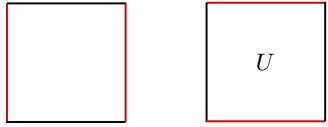
\begin{center}

\tikzset{every picture/.style={line width=0.75pt}} 

\begin{tikzpicture}[x=0.75pt,y=0.75pt,yscale=-1,xscale=1]

\draw [color={rgb, 255:red, 0; green, 0; blue, 0 }  ,draw opacity=1 ]   (230.07,140.3) -- (290.23,140.3) ;
\draw [color={rgb, 255:red, 208; green, 2; blue, 27 }  ,draw opacity=1 ]   (230.07,140.3) -- (230.07,200.47) ;
\draw [color={rgb, 255:red, 0; green, 0; blue, 0 }  ,draw opacity=1 ]   (229.73,200.13) -- (289.9,200.13) ;
\draw [color={rgb, 255:red, 208; green, 2; blue, 27 }  ,draw opacity=1 ]   (289.9,139.97) -- (289.9,200.13) ;
\draw [color={rgb, 255:red, 208; green, 2; blue, 27 }  ,draw opacity=1 ]   (330.73,139.97) -- (390.9,139.97) ;
\draw [color={rgb, 255:red, 0; green, 0; blue, 0 }  ,draw opacity=1 ]   (330.73,139.97) -- (330.73,200.13) ;
\draw [color={rgb, 255:red, 208; green, 2; blue, 27 }  ,draw opacity=1 ]   (330.4,199.8) -- (390.57,199.8) ;
\draw [color={rgb, 255:red, 0; green, 0; blue, 0 }  ,draw opacity=1 ]   (390.57,139.63) -- (390.57,199.8) ;

\draw (354,164.07) node [anchor=north west][inner sep=0.75pt]  [font=\footnotesize]  {$U$};

\end{tikzpicture}
\end{center}
\noindent Given a double category $C$ we write $\gamma C$ for the sub-double category of $C$ generated by squares of the above form. We call $\gamma C$ the \textbf{globularly generated piece} of $C$. $\gamma C$ is globularly generated, satisfies the equation 

\[H^*C=H^*\gamma C\]

\noindent and is contained in every sub-double category $D$ of $C$ satisfying the equation $H^*C=H^*D$. Moreover, a double category $C$ is globularly generated if and only if $C$ does not contain proper sub-double categories satisfying the above equation. Globularly generated double categories are thus minimal with respect to $H^*$. See \cite{Orendain3} for a functorial version of this statement. A consequence of the categorical version of this statement is that every double functor $F:C\to D$ from a globularly generated double category $C$ to a general double category $D$ factors uniquely through a double functor $C\to \gamma D$.

\subsection{Length 1}\label{ss:lengths}
\noindent Globularly generated double categories admit a helpful combinatorial description provided in the form of a filtration of their categories of squares. This defines a numerical invariant which we call the length. In this note we will only be interested in length 1 double categories. We refer the reader to \cite[Sec. 2.4]{OrendainMaldonado} for a general treatment. A globularly generated double category $C$ is of length 1, $\ell C=1$ in symbols, if every square in $C$ admits an expression as vertical composition of squares of the form

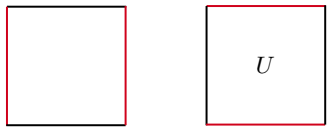
\begin{center}

\tikzset{every picture/.style={line width=0.75pt}} 

\begin{tikzpicture}[x=0.75pt,y=0.75pt,yscale=-1,xscale=1]

\draw [color={rgb, 255:red, 0; green, 0; blue, 0 }  ,draw opacity=1 ]   (230.07,140.3) -- (290.23,140.3) ;
\draw [color={rgb, 255:red, 208; green, 2; blue, 27 }  ,draw opacity=1 ]   (230.07,140.3) -- (230.07,200.47) ;
\draw [color={rgb, 255:red, 0; green, 0; blue, 0 }  ,draw opacity=1 ]   (229.73,200.13) -- (289.9,200.13) ;
\draw [color={rgb, 255:red, 208; green, 2; blue, 27 }  ,draw opacity=1 ]   (289.9,139.97) -- (289.9,200.13) ;
\draw [color={rgb, 255:red, 208; green, 2; blue, 27 }  ,draw opacity=1 ]   (330.73,139.97) -- (390.9,139.97) ;
\draw [color={rgb, 255:red, 0; green, 0; blue, 0 }  ,draw opacity=1 ]   (330.73,139.97) -- (330.73,200.13) ;
\draw [color={rgb, 255:red, 208; green, 2; blue, 27 }  ,draw opacity=1 ]   (330.4,199.8) -- (390.57,199.8) ;
\draw [color={rgb, 255:red, 0; green, 0; blue, 0 }  ,draw opacity=1 ]   (390.57,139.63) -- (390.57,199.8) ;

\draw (354,164.07) node [anchor=north west][inner sep=0.75pt]  [font=\footnotesize]  {$U$};

\end{tikzpicture}

\end{center}

\noindent i.e. if every square in $C$ is a vertical composite of horizontal identity and globular squares. More geometrically, if we regard the globular squares and horizontal identities of a double category $C$ as the simplest possible squares of $C$, then it is natural to represent globular and horizontal identity squares as squares marked by 0, i.e. as:

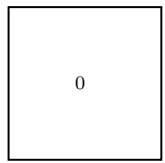
\begin{center}

\tikzset{every picture/.style={line width=0.75pt}} 

\begin{tikzpicture}[x=0.75pt,y=0.75pt,yscale=-1,xscale=1]

\draw   (292.77,151.6) -- (369.82,151.6) -- (369.82,228.65) -- (292.77,228.65) -- cycle ;

\draw (324.95,185.52) node [anchor=north west][inner sep=0.75pt]  [font=\scriptsize,xscale=0.9,yscale=0.9]  {$0$};

\end{tikzpicture}

\end{center}

\noindent A trivial, but important observation is that the collection of 0-marked squares is closed under horizontal composition. $C$ is of length 1 if every square in $C$ can be represented geometrically as:

\begin{center}

\tikzset{every picture/.style={line width=0.75pt}} 

\begin{tikzpicture}[x=0.75pt,y=0.75pt,yscale=-1,xscale=1]

\draw   (290.34,171.28) -- (369.16,171.28) -- (369.16,250.09) -- (290.34,250.09) -- cycle ;
\draw    (290.41,190.71) -- (369.3,190.41) ;
\draw    (290.64,210.71) -- (368.86,210.66) ;
\draw    (290.19,230.49) -- (369.3,230.41) ;

\draw (324.45,177.54) node [anchor=north west][inner sep=0.75pt]  [font=\scriptsize,xscale=0.9,yscale=0.9]  {$0$};
\draw (324.34,195.76) node [anchor=north west][inner sep=0.75pt]  [font=\scriptsize,xscale=0.9,yscale=0.9]  {$0$};
\draw (323.9,236.43) node [anchor=north west][inner sep=0.75pt]  [font=\scriptsize,xscale=0.9,yscale=0.9]  {$0$};
\draw (333.41,211.46) node [anchor=north west][inner sep=0.75pt]  [font=\scriptsize,rotate=-90.3,xscale=0.9,yscale=0.9]  {$\cdots $};

\end{tikzpicture}

\end{center}

\noindent In general, given two horizontally composable squares admitting subdivisions as above $\varphi,\psi$, it might be the case that for every choice of vertical subdivision of $\varphi$ and $\psi$ into marked 0 squares, the horizontal composition $\varphi\boxbar\psi$ looks like

\begin{center}

\tikzset{every picture/.style={line width=0.75pt}} 

\begin{tikzpicture}[x=0.75pt,y=0.75pt,yscale=-1,xscale=1]

\draw   (271.59,161) -- (350.41,161) -- (350.41,239.82) -- (271.59,239.82) -- cycle ;
\draw    (271.66,180.43) -- (350.55,180.14) ;
\draw    (271.89,197.58) -- (350.11,197.53) ;
\draw    (271.44,220.21) -- (350.55,220.14) ;
\draw   (351.09,161) -- (429.91,161) -- (429.91,239.82) -- (351.09,239.82) -- cycle ;
\draw    (351.16,190.15) -- (430.27,190.7) ;
\draw    (349.8,211.35) -- (429.85,211.23) ;

\draw (305.7,167.26) node [anchor=north west][inner sep=0.75pt]  [font=\scriptsize,xscale=0.9,yscale=0.9]  {$0$};
\draw (305.59,183.48) node [anchor=north west][inner sep=0.75pt]  [font=\scriptsize,xscale=0.9,yscale=0.9]  {$0$};
\draw (305.15,226.15) node [anchor=north west][inner sep=0.75pt]  [font=\scriptsize,xscale=0.9,yscale=0.9]  {$0$};
\draw (314.66,201.18) node [anchor=north west][inner sep=0.75pt]  [font=\scriptsize,rotate=-90.3,xscale=0.9,yscale=0.9]  {$\cdots $};
\draw (384.63,171.83) node [anchor=north west][inner sep=0.75pt]  [font=\scriptsize,xscale=0.9,yscale=0.9]  {$0$};
\draw (385.22,222.15) node [anchor=north west][inner sep=0.75pt]  [font=\scriptsize,xscale=0.9,yscale=0.9]  {$0$};
\draw (394.44,192.61) node [anchor=north west][inner sep=0.75pt]  [font=\scriptsize,rotate=-90.3,xscale=0.9,yscale=0.9]  {$\cdots $};

\end{tikzpicture}

\end{center}

\noindent See \cite[Example 4.1]{Orendain2} for an explicit example. $C$ is of length 1, precisely when the above does not happen, i.e. when given two horizontally compatible squares $\varphi,\psi$ admitting decompositions as vertical compositions of squares marked with 0, then we can always find vertical subdivisions making $\varphi\boxbar\psi$ look like:

\begin{center}

\tikzset{every picture/.style={line width=0.75pt}} 

\begin{tikzpicture}[x=0.75pt,y=0.75pt,yscale=-1,xscale=1]

\draw   (251.59,141) -- (330.41,141) -- (330.41,219.82) -- (251.59,219.82) -- cycle ;
\draw    (251.66,160.43) -- (330.55,160.14) ;
\draw    (251.89,180.43) -- (330.11,180.38) ;
\draw    (251.44,200.21) -- (330.55,200.14) ;
\draw   (331.09,141) -- (409.91,141) -- (409.91,219.82) -- (331.09,219.82) -- cycle ;
\draw    (331.16,160.43) -- (410.05,160.14) ;
\draw    (331.39,180.43) -- (409.61,180.38) ;
\draw    (330.94,200.21) -- (410.05,200.14) ;

\draw (285.7,147.26) node [anchor=north west][inner sep=0.75pt]  [font=\scriptsize,xscale=0.9,yscale=0.9]  {$0$};
\draw (285.59,165.48) node [anchor=north west][inner sep=0.75pt]  [font=\scriptsize,xscale=0.9,yscale=0.9]  {$0$};
\draw (285.15,206.15) node [anchor=north west][inner sep=0.75pt]  [font=\scriptsize,xscale=0.9,yscale=0.9]  {$0$};
\draw (294.66,181.18) node [anchor=north west][inner sep=0.75pt]  [font=\scriptsize,rotate=-90.3,xscale=0.9,yscale=0.9]  {$\cdots $};
\draw (365.2,147.26) node [anchor=north west][inner sep=0.75pt]  [font=\scriptsize,xscale=0.9,yscale=0.9]  {$0$};
\draw (365.09,165.48) node [anchor=north west][inner sep=0.75pt]  [font=\scriptsize,xscale=0.9,yscale=0.9]  {$0$};
\draw (364.65,206.15) node [anchor=north west][inner sep=0.75pt]  [font=\scriptsize,xscale=0.9,yscale=0.9]  {$0$};
\draw (374.16,181.18) node [anchor=north west][inner sep=0.75pt]  [font=\scriptsize,rotate=-90.3,xscale=0.9,yscale=0.9]  {$\cdots $};

\end{tikzpicture}

\end{center}

\noindent In that case we can use the exchange identity to re-arrange the above composition into a vertical subdivision of 0-marked squares. Squares in a globularly generated double category of length 1 thus admit a simple expression in terms of globular and unit squares. It is thus desirable to have criteria for knowing when a given double category is of length 1.

\subsection{$\pi_2$-opindexings}\label{ss:piindexings}

\begin{definition}\label{def:opindexing}
Let $(B^*,B)$ be a decorated 2-category. A \textbf{$\pi_2$-indexing} on $(B^*,B)$ is a functor
\[\indFunct\]
such that for every object $a$ in $\stB$, the monoid $\Phi(a)$ is the commutative monoid $\pi_2(B,a)$ of a, in $B$. A \textbf{$\pi_2$-opindexing} on $(B^*,B)$ is an indexing on $(B^{\ast op},B)$. 
\end{definition}

\noindent Examples of $\pi_2$-indexings and $\pi_2$-opindexings can be found in \cite{OrendainMaldonado}. Given a $\pi_2$-indexing $\Phi$ on a decorated bicategory $\decB$, for every morphsim $f:a\to b$ in $B^\ast$, and red boundary square $\varphi\in\pi_2(B,a)$, the $\pi_2$-indexing $\Phi$ functorially provides us with a new red boundary square $\Phi_f(\varphi)\in\pi_2(B,b)$. We interpret this as an operation of sliding the red bounday square $\varphi$ \textit{down} the morphism $f$. In the case in which $\decB=H^\ast C$ for a double category $C$, then we can identify $f$ with its unit square $U_f$ and in that case we can interpret $\Phi$ as a functorial operation of sliding red boundary squares down unit squares. Similarly, $\pi_2$-opindexings provide a functorial operation of sliding red boundary squares \textit{up} decorating morphisms/unit squares. In \cite{OrendainMaldonado} the authors focus on $\pi_2$-indexings. We will, in this note focus mostly on $\pi_2$-opindexings. Our main interest on $\pi_2$-opindexings comes from the following result, saying that we can associate, to every $\pi_2$-indexing or $\pi_2$-opindexing on a decorated bicategory $\decB$, a length 1 internalization of $\decB$.
\begin{thm}\label{thm:mainindexings}
    Let $\decB$ be a decorated 2-category. For every $\pi_2$-indexing or $\pi_2$-opindexing $\Phi$ on $\decB$, there exists an internalization $\crossb$ of $\decB$ such that $\ell(\crossb)=1$.
\end{thm}
\noindent We briefly review the construction of the double category $\crossb$ from the data of a $\pi_2$-opindexing appearing in Theorem \ref{thm:mainindexings}. In \cite{OrendainMaldonado} this discussion is done explicitly for $\pi_2$-indexings. The idea behind the construction of $\crossb$ is to reverse engineer double categories where every length 1 square admits a vertical subdivision of the form:
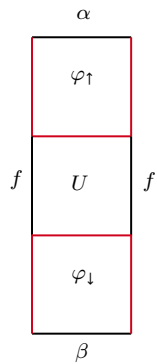
\begin{center}

\tikzset{every picture/.style={line width=0.75pt}} 

\begin{tikzpicture}[x=0.75pt,y=0.75pt,yscale=-1,xscale=1]

\draw [color={rgb, 255:red, 0; green, 0; blue, 0 }  ,draw opacity=1 ]   (300,69.4) -- (350,69.4) ;
\draw [color={rgb, 255:red, 208; green, 2; blue, 27 }  ,draw opacity=1 ]   (300,119.4) -- (350,119.4) ;
\draw [color={rgb, 255:red, 208; green, 2; blue, 27 }  ,draw opacity=1 ]   (300,69.4) -- (300,119.4) ;
\draw [color={rgb, 255:red, 208; green, 2; blue, 27 }  ,draw opacity=1 ]   (350,69.4) -- (350,119.4) ;
\draw [color={rgb, 255:red, 0; green, 0; blue, 0 }  ,draw opacity=1 ]   (300,119.4) -- (300,169.4) ;
\draw [color={rgb, 255:red, 0; green, 0; blue, 0 }  ,draw opacity=1 ]   (350,119.4) -- (350,169.4) ;
\draw [color={rgb, 255:red, 208; green, 2; blue, 27 }  ,draw opacity=1 ]   (300,169.4) -- (350,169.4) ;
\draw [color={rgb, 255:red, 0; green, 0; blue, 0 }  ,draw opacity=1 ]   (300,219) -- (350,219) ;
\draw [color={rgb, 255:red, 208; green, 2; blue, 27 }  ,draw opacity=1 ]   (300,169) -- (300,219) ;
\draw [color={rgb, 255:red, 208; green, 2; blue, 27 }  ,draw opacity=1 ]   (350,169) -- (350,219) ;

\draw (354.36,135.84) node [anchor=north west][inner sep=0.75pt]  [font=\scriptsize]  {$f$};
\draw (287.36,134.84) node [anchor=north west][inner sep=0.75pt]  [font=\scriptsize]  {$f$};
\draw (321.01,54.3) node [anchor=north west][inner sep=0.75pt]  [font=\scriptsize]  {$\alpha $};
\draw (318.36,185.44) node [anchor=north west][inner sep=0.75pt]  [font=\scriptsize]  {$\varphi _{\downarrow }$};
\draw (318.21,137.5) node [anchor=north west][inner sep=0.75pt]  [font=\scriptsize]  {$U$};
\draw (320.21,222.6) node [anchor=north west][inner sep=0.75pt]  [font=\scriptsize]  {$\beta $};
\draw (317.96,84.64) node [anchor=north west][inner sep=0.75pt]  [font=\scriptsize]  {$\varphi _{\uparrow }$};

\end{tikzpicture}

\end{center}
\noindent A subdivision of a square $\varphi$ of the above form is called a \textbf{canonical decomposition} of $\varphi$, and any square in a double category admitting a canonical decomposition is called a \textbf{canonical square}. In \cite[Lemma 3.9]{OrendainMaldonado} it is proven that if every length 1 square in a double category $C$ is canonical, then $\ell C=1$.  Given a $\pi_2$-opindexing $\Phi$ on a decorated bicategory $\decB$, a square in $\crossb$ is either a 2-cell in $B$ or a triple $(\varphi_\downarrow,f,\varphi_\uparrow)$, where $f$ is a morphism in $B^\ast$ from an object $a$ to an object $b$, and where $\varphi_\downarrow,\varphi_\uparrow$ are 2-cells in $B$, $\varphi_\uparrow:\alpha\Rightarrow id_a$ and $\varphi_\downarrow:id_v\Rightarrow \beta$ for some 1-cells $\alpha:a\to a$ and $\beta:b\to b$ in $B$. We identify triples of the form $(\varphi_\downarrow,id,\varphi_\uparrow)$ with 2-cells $\varphi_\downarrow\boxminus \varphi_\uparrow$ in $B$. Given a morphism $f$ in $B^\ast$ the unit square $U_f$ in $\crossb$ is the square $(id,f,id)$. Moreover, we identify the 2-cells $\varphi_\downarrow$ and $\varphi_\uparrow$ with $(\varphi_\downarrow,id,id)$ and $(id,id,\varphi_\uparrow)$. Every nonglobular square $(\varphi_\downarrow,f,\varphi_\uparrow)$ in $\crossb$ can thus be identified with a canonical square $\varphi_\downarrow\boxminus U_f\boxminus \varphi_\uparrow$ and can thus be represented pictorially through a subdivision as above. We will generally use a different pictorial representation for squares in $\crossb$ in order to distinguish squares in a general double category $C$ with squares in double categories of the form $\crossb$. We will represent a nonglobular square $(\varphi_\downarrow,f,\varphi_\uparrow)$ in a double category of the form $\crossb$ as
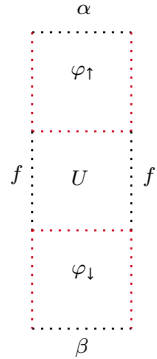
\begin{center}

\tikzset{every picture/.style={line width=0.75pt}} 

\begin{tikzpicture}[x=0.75pt,y=0.75pt,yscale=-1,xscale=1]

\draw [color={rgb, 255:red, 0; green, 0; blue, 0 }  ,draw opacity=1 ] [dash pattern={on 0.84pt off 2.51pt}]  (280,49.4) -- (330,49.4) ;
\draw [color={rgb, 255:red, 208; green, 2; blue, 27 }  ,draw opacity=1 ] [dash pattern={on 0.84pt off 2.51pt}]  (280,99.4) -- (330,99.4) ;
\draw [color={rgb, 255:red, 208; green, 2; blue, 27 }  ,draw opacity=1 ] [dash pattern={on 0.84pt off 2.51pt}]  (280,49.4) -- (280,99.4) ;
\draw [color={rgb, 255:red, 208; green, 2; blue, 27 }  ,draw opacity=1 ] [dash pattern={on 0.84pt off 2.51pt}]  (330,49.4) -- (330,99.4) ;
\draw [color={rgb, 255:red, 0; green, 0; blue, 0 }  ,draw opacity=1 ] [dash pattern={on 0.84pt off 2.51pt}]  (280,99.4) -- (280,149.4) ;
\draw [color={rgb, 255:red, 0; green, 0; blue, 0 }  ,draw opacity=1 ] [dash pattern={on 0.84pt off 2.51pt}]  (330,99.4) -- (330,149.4) ;
\draw [color={rgb, 255:red, 208; green, 2; blue, 27 }  ,draw opacity=1 ] [dash pattern={on 0.84pt off 2.51pt}]  (280,149.4) -- (330,149.4) ;
\draw [color={rgb, 255:red, 0; green, 0; blue, 0 }  ,draw opacity=1 ] [dash pattern={on 0.84pt off 2.51pt}]  (280,199) -- (330,199) ;
\draw [color={rgb, 255:red, 208; green, 2; blue, 27 }  ,draw opacity=1 ] [dash pattern={on 0.84pt off 2.51pt}]  (280,149) -- (280,199) ;
\draw [color={rgb, 255:red, 208; green, 2; blue, 27 }  ,draw opacity=1 ] [dash pattern={on 0.84pt off 2.51pt}]  (330,149) -- (330,199) ;

\draw (334.36,115.84) node [anchor=north west][inner sep=0.75pt]  [font=\scriptsize]  {$f$};
\draw (267.36,114.84) node [anchor=north west][inner sep=0.75pt]  [font=\scriptsize]  {$f$};
\draw (301.01,34.3) node [anchor=north west][inner sep=0.75pt]  [font=\scriptsize]  {$\alpha $};
\draw (298.36,165.44) node [anchor=north west][inner sep=0.75pt]  [font=\scriptsize]  {$\varphi _{\downarrow }$};
\draw (298.21,117.5) node [anchor=north west][inner sep=0.75pt]  [font=\scriptsize]  {$U$};
\draw (300.21,202.6) node [anchor=north west][inner sep=0.75pt]  [font=\scriptsize]  {$\beta $};
\draw (297.96,64.64) node [anchor=north west][inner sep=0.75pt]  [font=\scriptsize]  {$\varphi _{\uparrow }$};

\end{tikzpicture}

\end{center}
i.e. we will distinguish squares in double categories of the form $\crossb$ from squares in general double categories by drawing squares in $\crossb$ with dotted lines. Nonglobular squares in $\crossb$ are thus by design all canonical, and thus $\ell(\crossb)=1$. The following observation follows directly from the construction of $\crossb$.

\begin{obs}\label{obs:squarescrosspsame}
    Let $\decB$ be a decorated bicategory. Let $\Phi$ be a $\pi_2$-opindexing. The only relation between squares in the double category $\crossb$ is the following: Two canonical squares
    
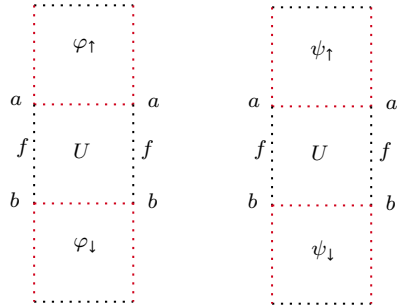
\begin{center}

\tikzset{every picture/.style={line width=0.75pt}} 



    \end{center}
    \noindent are equal if and only if there exists a red-boundary square 
    
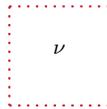
\begin{center}

\tikzset{every picture/.style={line width=0.75pt}} 

%

    \end{center}
    \noindent in $\pi_2(B,b)$ such that there are equalities of squares 
    \begin{center}

\tikzset{every picture/.style={line width=0.75pt}} 

%

    \end{center}
    \noindent The fact that these relations hold follows from the following pictorial equation
    \begin{center}

\tikzset{every picture/.style={line width=0.75pt}} 

%

    \end{center}
    It is easily seen, from the construction presented in \cite{OrendainMaldonado} that this is indeed the only relation between squares in $\crossb$.
\end{obs}

\noindent We now recall what the horizontal and vertical composition operations in $\crossb$ look like geometrically. \textbf{Horizontal composition:} Given horizontally composable squares in $\crossb$
\begin{center}

\tikzset{every picture/.style={line width=0.75pt}} 

%

\end{center}
\noindent their horizontal composition is defined as
\begin{center}

\tikzset{every picture/.style={line width=0.75pt}} 

%

\end{center}
\noindent Geometrically, the above canonical square is the composite
\begin{center}

\tikzset{every picture/.style={line width=0.75pt}} 

%

\end{center}
\noindent \textbf{Vertical composition:} The vertical composition of two vertically composable squares in canonical form
\begin{center}

\tikzset{every picture/.style={line width=0.75pt}} 

%

\end{center}
\noindent is the square in canonical form
\begin{center}

\tikzset{every picture/.style={line width=0.75pt}} 

%

\end{center}

\subsection{Framed bicategories}\label{ss:framed}
\noindent A \textbf{framed bicategory} is a double category such that the left-right frame functor $L\times R:C_1\to C_0^2$ is a bifibration \cite{ShulmanFramed,ShulmanDerived}. The condition of a double category $C$ being a framed bicategory provides extension and restriction of horizontal arrows along vertical arrows. Framed bicategories are essentially the same as proarrow equipments, as defined by Wood in \cite{Wood1,Wood2}, see \cite[Appendix C]{ShulmanFramed}. In a framed bicategory $C$, given a hollow niche
\begin{center}
\tikzset{every picture/.style={line width=0.75pt}} 

\begin{tikzpicture}[x=0.75pt,y=0.75pt,yscale=-1,xscale=1]

\draw [color={rgb, 255:red, 0; green, 0; blue, 0 }  ,draw opacity=1 ]   (241,190.6) -- (291,190.6) ;
\draw [color={rgb, 255:red, 0; green, 0; blue, 0 }  ,draw opacity=1 ]   (241,140.6) -- (241,190.6) ;
\draw [color={rgb, 255:red, 0; green, 0; blue, 0 }  ,draw opacity=1 ]   (291,140.6) -- (291,190.6) ;

\draw (221.6,158) node [anchor=north west][inner sep=0.75pt]  [font=\scriptsize]  {$f$};
\draw (297.6,158.2) node [anchor=north west][inner sep=0.75pt]  [font=\scriptsize]  {$g$};
\draw (260.8,192.8) node [anchor=north west][inner sep=0.75pt]  [font=\scriptsize]  {$\alpha $};

\end{tikzpicture}
\end{center}
\noindent there exists a Cartesian square
\begin{center}

\tikzset{every picture/.style={line width=0.75pt}} 

\begin{tikzpicture}[x=0.75pt,y=0.75pt,yscale=-1,xscale=1]

\draw [color={rgb, 255:red, 0; green, 0; blue, 0 }  ,draw opacity=1 ]   (261,210.6) -- (311,210.6) ;
\draw [color={rgb, 255:red, 0; green, 0; blue, 0 }  ,draw opacity=1 ]   (261,160.6) -- (261,210.6) ;
\draw [color={rgb, 255:red, 0; green, 0; blue, 0 }  ,draw opacity=1 ]   (311,160.6) -- (311,210.6) ;
\draw [color={rgb, 255:red, 0; green, 0; blue, 0 }  ,draw opacity=1 ]   (261,160.6) -- (311,160.6) ;

\draw (241.6,178) node [anchor=north west][inner sep=0.75pt]  [font=\scriptsize]  {$f$};
\draw (317.6,178.2) node [anchor=north west][inner sep=0.75pt]  [font=\scriptsize]  {$g$};
\draw (280.8,212.8) node [anchor=north west][inner sep=0.75pt]  [font=\scriptsize]  {$\alpha $};
\draw (271.47,142.4) node [anchor=north west][inner sep=0.75pt]  [font=\scriptsize]  {$f^{\ast } \alpha g^{\ast }$};
\draw (274.33,179.67) node [anchor=north west][inner sep=0.75pt]  [font=\scriptsize] [align=left] {Cart};

\end{tikzpicture}

\end{center}
\noindent filling the niche. A Cartesian square, as above, provides a way of changing basis for the horizontal morphism $\alpha$ along restriction with respect to the vertical morphisms $f$ and $g$. Given a Cartesian square as above, every other square of the form
\begin{center}

\tikzset{every picture/.style={line width=0.75pt}} 

\begin{tikzpicture}[x=0.75pt,y=0.75pt,yscale=-1,xscale=1]

\draw [color={rgb, 255:red, 0; green, 0; blue, 0 }  ,draw opacity=1 ]   (261,179.6) -- (311,179.6) ;
\draw [color={rgb, 255:red, 0; green, 0; blue, 0 }  ,draw opacity=1 ]   (261,129.6) -- (261,179.6) ;
\draw [color={rgb, 255:red, 0; green, 0; blue, 0 }  ,draw opacity=1 ]   (311,129.6) -- (311,179.6) ;
\draw [color={rgb, 255:red, 0; green, 0; blue, 0 }  ,draw opacity=1 ]   (261,129.6) -- (311,129.6) ;

\draw (241.6,147) node [anchor=north west][inner sep=0.75pt]  [font=\scriptsize]  {$fh$};
\draw (317.6,147.2) node [anchor=north west][inner sep=0.75pt]  [font=\scriptsize]  {$gk$};
\draw (280.8,181.8) node [anchor=north west][inner sep=0.75pt]  [font=\scriptsize]  {$\alpha $};
\draw (281.47,112.73) node [anchor=north west][inner sep=0.75pt]  [font=\scriptsize]  {$\beta $};
\draw (278.61,149.3) node [anchor=north west][inner sep=0.75pt]  [font=\scriptsize]  {$\Phi $};

\end{tikzpicture}

\end{center}
\noindent factors uniquely as
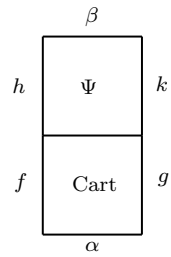
\begin{center}

\tikzset{every picture/.style={line width=0.75pt}} 

\begin{tikzpicture}[x=0.75pt,y=0.75pt,yscale=-1,xscale=1]

\draw [color={rgb, 255:red, 0; green, 0; blue, 0 }  ,draw opacity=1 ]   (281,199.6) -- (331,199.6) ;
\draw [color={rgb, 255:red, 0; green, 0; blue, 0 }  ,draw opacity=1 ]   (281,149.6) -- (281,199.6) ;
\draw [color={rgb, 255:red, 0; green, 0; blue, 0 }  ,draw opacity=1 ]   (331,149.6) -- (331,199.6) ;
\draw [color={rgb, 255:red, 0; green, 0; blue, 0 }  ,draw opacity=1 ]   (281,149.6) -- (331,149.6) ;
\draw [color={rgb, 255:red, 0; green, 0; blue, 0 }  ,draw opacity=1 ]   (281.02,99.65) -- (281.02,149.65) ;
\draw [color={rgb, 255:red, 0; green, 0; blue, 0 }  ,draw opacity=1 ]   (331.02,99.65) -- (331.02,149.65) ;
\draw [color={rgb, 255:red, 0; green, 0; blue, 0 }  ,draw opacity=1 ]   (281.02,99.65) -- (331.02,99.65) ;

\draw (265.14,167.31) node [anchor=north west][inner sep=0.75pt]  [font=\scriptsize]  {$f$};
\draw (337.6,167.2) node [anchor=north west][inner sep=0.75pt]  [font=\scriptsize]  {$g$};
\draw (300.8,201.8) node [anchor=north west][inner sep=0.75pt]  [font=\scriptsize]  {$\alpha $};
\draw (300.85,83.19) node [anchor=north west][inner sep=0.75pt]  [font=\scriptsize]  {$\beta $};
\draw (294.33,168.67) node [anchor=north west][inner sep=0.75pt]  [font=\scriptsize] [align=left] {Cart};
\draw (264.22,119.02) node [anchor=north west][inner sep=0.75pt]  [font=\scriptsize]  {$h$};
\draw (336.52,118.25) node [anchor=north west][inner sep=0.75pt]  [font=\scriptsize]  {$k$};
\draw (298.24,119.78) node [anchor=north west][inner sep=0.75pt]  [font=\scriptsize]  {$\Psi $};

\end{tikzpicture}

\end{center}
\noindent for some uniquely determined square $\Psi$. In particular, every square
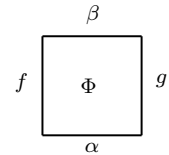
\begin{center}

\tikzset{every picture/.style={line width=0.75pt}} 

\begin{tikzpicture}[x=0.75pt,y=0.75pt,yscale=-1,xscale=1]

\draw [color={rgb, 255:red, 0; green, 0; blue, 0 }  ,draw opacity=1 ]   (281,199.6) -- (331,199.6) ;
\draw [color={rgb, 255:red, 0; green, 0; blue, 0 }  ,draw opacity=1 ]   (281,149.6) -- (281,199.6) ;
\draw [color={rgb, 255:red, 0; green, 0; blue, 0 }  ,draw opacity=1 ]   (331,149.6) -- (331,199.6) ;
\draw [color={rgb, 255:red, 0; green, 0; blue, 0 }  ,draw opacity=1 ]   (281,149.6) -- (331,149.6) ;

\draw (265.6,167) node [anchor=north west][inner sep=0.75pt]  [font=\scriptsize]  {$f$};
\draw (336.6,167.2) node [anchor=north west][inner sep=0.75pt]  [font=\scriptsize]  {$g$};
\draw (300.8,201.8) node [anchor=north west][inner sep=0.75pt]  [font=\scriptsize]  {$\alpha $};
\draw (301.47,132.73) node [anchor=north west][inner sep=0.75pt]  [font=\scriptsize]  {$\beta $};
\draw (298.61,169.3) node [anchor=north west][inner sep=0.75pt]  [font=\scriptsize]  {$\Phi $};

\end{tikzpicture}

\end{center}
\noindent Factors uniquely as
\begin{center}

\tikzset{every picture/.style={line width=0.75pt}} 

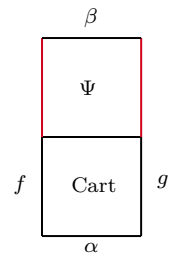
\begin{tikzpicture}[x=0.75pt,y=0.75pt,yscale=-1,xscale=1]

\draw [color={rgb, 255:red, 0; green, 0; blue, 0 }  ,draw opacity=1 ]   (301,219.2) -- (351,219.2) ;
\draw [color={rgb, 255:red, 0; green, 0; blue, 0 }  ,draw opacity=1 ]   (301,169.2) -- (301,219.2) ;
\draw [color={rgb, 255:red, 0; green, 0; blue, 0 }  ,draw opacity=1 ]   (351,169.2) -- (351,219.2) ;
\draw [color={rgb, 255:red, 0; green, 0; blue, 0 }  ,draw opacity=1 ]   (301,169.2) -- (351,169.2) ;
\draw [color={rgb, 255:red, 208; green, 2; blue, 27 }  ,draw opacity=1 ]   (301.02,119.25) -- (301.02,169.25) ;
\draw [color={rgb, 255:red, 208; green, 2; blue, 27 }  ,draw opacity=1 ]   (351.02,119.25) -- (351.02,169.25) ;
\draw [color={rgb, 255:red, 0; green, 0; blue, 0 }  ,draw opacity=1 ]   (301.02,119.25) -- (351.02,119.25) ;

\draw (285.14,186.91) node [anchor=north west][inner sep=0.75pt]  [font=\scriptsize]  {$f$};
\draw (357.6,186.8) node [anchor=north west][inner sep=0.75pt]  [font=\scriptsize]  {$g$};
\draw (320.8,221.4) node [anchor=north west][inner sep=0.75pt]  [font=\scriptsize]  {$\alpha $};
\draw (320.85,102.79) node [anchor=north west][inner sep=0.75pt]  [font=\scriptsize]  {$\beta $};
\draw (314.33,188.27) node [anchor=north west][inner sep=0.75pt]  [font=\scriptsize] [align=left] {Cart};
\draw (318.24,139.38) node [anchor=north west][inner sep=0.75pt]  [font=\scriptsize]  {$\Psi $};

\end{tikzpicture}

\end{center}
\noindent A choice of cartesian squares for niches as above is a \textbf{cleavage} for $L\times R$. In case a cleavage for $L\times R$ has been chosen, we will say that $L\times R$ is a cloven fibration. A cloven fibration is \textbf{normal} if the square
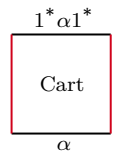
\begin{center}

\tikzset{every picture/.style={line width=0.75pt}} 

\begin{tikzpicture}[x=0.75pt,y=0.75pt,yscale=-1,xscale=1]

\draw [color={rgb, 255:red, 0; green, 0; blue, 0 }  ,draw opacity=1 ]   (180,180.6) -- (230,180.6) ;
\draw [color={rgb, 255:red, 0; green, 0; blue, 0 }  ,draw opacity=1 ]   (180,230.6) -- (230,230.6) ;
\draw [color={rgb, 255:red, 208; green, 2; blue, 27 }  ,draw opacity=1 ]   (180,180.6) -- (180,230.6) ;
\draw [color={rgb, 255:red, 208; green, 2; blue, 27 }  ,draw opacity=1 ]   (230,180.6) -- (230,230.6) ;

\draw (189.69,165.71) node [anchor=north west][inner sep=0.75pt]  [font=\scriptsize]  {$1^{*} \alpha 1^{\ast }$};
\draw (201.36,232.71) node [anchor=north west][inner sep=0.75pt]  [font=\scriptsize]  {$\alpha $};
\draw (192.67,200.33) node [anchor=north west][inner sep=0.75pt]  [font=\scriptsize] [align=left] {Cart};

\end{tikzpicture}

\end{center}
\noindent is the horizontal identity square
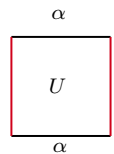
\begin{center}

\tikzset{every picture/.style={line width=0.75pt}} 

\begin{tikzpicture}[x=0.75pt,y=0.75pt,yscale=-1,xscale=1]

\draw [color={rgb, 255:red, 0; green, 0; blue, 0 }  ,draw opacity=1 ]   (160,160.6) -- (210,160.6) ;
\draw [color={rgb, 255:red, 0; green, 0; blue, 0 }  ,draw opacity=1 ]   (160,210.6) -- (210,210.6) ;
\draw [color={rgb, 255:red, 208; green, 2; blue, 27 }  ,draw opacity=1 ]   (160,160.6) -- (160,210.6) ;
\draw [color={rgb, 255:red, 208; green, 2; blue, 27 }  ,draw opacity=1 ]   (210,160.6) -- (210,210.6) ;

\draw (178.69,145.71) node [anchor=north west][inner sep=0.75pt]  [font=\scriptsize]  {$\alpha $};
\draw (179.36,212.71) node [anchor=north west][inner sep=0.75pt]  [font=\scriptsize]  {$\alpha $};
\draw (177.36,180.37) node [anchor=north west][inner sep=0.75pt]  [font=\scriptsize]  {$U$};

\end{tikzpicture}

\end{center}
\noindent and is \textbf{split} if the Cartesian square
\begin{center}

\tikzset{every picture/.style={line width=0.75pt}} 

\begin{tikzpicture}[x=0.75pt,y=0.75pt,yscale=-1,xscale=1]

\draw [color={rgb, 255:red, 0; green, 0; blue, 0 }  ,draw opacity=1 ]   (200,200.6) -- (250,200.6) ;
\draw [color={rgb, 255:red, 0; green, 0; blue, 0 }  ,draw opacity=1 ]   (200,250.6) -- (250,250.6) ;
\draw [color={rgb, 255:red, 0; green, 0; blue, 0 }  ,draw opacity=1 ]   (200,200.6) -- (200,250.6) ;
\draw [color={rgb, 255:red, 0; green, 0; blue, 0 }  ,draw opacity=1 ]   (250,200.6) -- (250,250.6) ;

\draw (196.69,182.71) node [anchor=north west][inner sep=0.75pt]  [font=\scriptsize]  {$( fh)^{*} \alpha ( gk)^{\ast }$};
\draw (221.36,252.71) node [anchor=north west][inner sep=0.75pt]  [font=\scriptsize]  {$\alpha $};
\draw (212.67,220.33) node [anchor=north west][inner sep=0.75pt]  [font=\scriptsize] [align=left] {Cart};
\draw (255.36,218.4) node [anchor=north west][inner sep=0.75pt]  [font=\scriptsize]  {$gk$};
\draw (182.76,218) node [anchor=north west][inner sep=0.75pt]  [font=\scriptsize]  {$fh$};

\end{tikzpicture}

\end{center}
\noindent is the composite
\begin{center}

\tikzset{every picture/.style={line width=0.75pt}} 

\begin{tikzpicture}[x=0.75pt,y=0.75pt,yscale=-1,xscale=1]

\draw [color={rgb, 255:red, 0; green, 0; blue, 0 }  ,draw opacity=1 ]   (301,219.6) -- (351,219.6) ;
\draw [color={rgb, 255:red, 0; green, 0; blue, 0 }  ,draw opacity=1 ]   (301,169.6) -- (301,219.6) ;
\draw [color={rgb, 255:red, 0; green, 0; blue, 0 }  ,draw opacity=1 ]   (351,169.6) -- (351,219.6) ;
\draw [color={rgb, 255:red, 0; green, 0; blue, 0 }  ,draw opacity=1 ]   (301,169.6) -- (351,169.6) ;
\draw [color={rgb, 255:red, 0; green, 0; blue, 0 }  ,draw opacity=1 ]   (301.02,119.65) -- (301.02,169.65) ;
\draw [color={rgb, 255:red, 0; green, 0; blue, 0 }  ,draw opacity=1 ]   (351.02,119.65) -- (351.02,169.65) ;
\draw [color={rgb, 255:red, 0; green, 0; blue, 0 }  ,draw opacity=1 ]   (301.02,119.65) -- (351.02,119.65) ;

\draw (285.14,187.31) node [anchor=north west][inner sep=0.75pt]  [font=\scriptsize]  {$f$};
\draw (357.6,187.2) node [anchor=north west][inner sep=0.75pt]  [font=\scriptsize]  {$g$};
\draw (320.8,221.8) node [anchor=north west][inner sep=0.75pt]  [font=\scriptsize]  {$\alpha $};
\draw (314.33,188.67) node [anchor=north west][inner sep=0.75pt]  [font=\scriptsize] [align=left] {Cart};
\draw (284.22,139.02) node [anchor=north west][inner sep=0.75pt]  [font=\scriptsize]  {$h$};
\draw (356.52,138.25) node [anchor=north west][inner sep=0.75pt]  [font=\scriptsize]  {$k$};
\draw (314.33,139.33) node [anchor=north west][inner sep=0.75pt]  [font=\scriptsize] [align=left] {Cart};
\draw (298,100.11) node [anchor=north west][inner sep=0.75pt]  [font=\scriptsize]  {$( fh)^{*} \alpha ( gk)^{\ast }$};

\end{tikzpicture}

\end{center}

\noindent We will assume, throughout the paper, that framed bicategories $C$ are such that \textbf{$L\times R$ is a normal, cloven fibration}.

\section{Double categories inducing $\pi_2$-idexings}\label{s:inducingpi2index}
\noindent In this section we prove that the construction of the length 1 internalization $\crossb$ associated to a $\pi_2$-indexing/$\pi_2$-opindexing $\Phi$ is universal in the category of double categories inducing $\Phi$. We will do this for $\pi_2$-opindexings. The arguments for the $\pi_2$-indexing case are completely analogous. The main argument is based on the fact that $\pi_2$-opindexings really are structures associated to internalizations of decorated bicategories. Let $C$ be a double category. Suppose $C$ admits an operation of sliding red boundary squares up unit squares, in the sense that for every vertical composition of the form
\begin{center}

\tikzset{every picture/.style={line width=0.75pt}} 

\begin{tikzpicture}[x=0.75pt,y=0.75pt,yscale=-1,xscale=1]

\draw [color={rgb, 255:red, 208; green, 2; blue, 27 }  ,draw opacity=1 ]   (180,181.6) -- (230,181.6) ;
\draw [color={rgb, 255:red, 208; green, 2; blue, 27 }  ,draw opacity=1 ]   (180,231.6) -- (230,231.6) ;
\draw    (180,181.6) -- (180,231.6) ;
\draw    (230,181.6) -- (230,231.6) ;
\draw [color={rgb, 255:red, 208; green, 2; blue, 27 }  ,draw opacity=1 ]   (180,281.6) -- (230,281.6) ;
\draw [color={rgb, 255:red, 208; green, 2; blue, 27 }  ,draw opacity=1 ]   (180,231.6) -- (180,281.6) ;
\draw [color={rgb, 255:red, 208; green, 2; blue, 27 }  ,draw opacity=1 ]   (230,231.6) -- (230,281.6) ;

\draw (232.36,200.04) node [anchor=north west][inner sep=0.75pt]  [font=\scriptsize]  {$f$};
\draw (169.36,200.04) node [anchor=north west][inner sep=0.75pt]  [font=\scriptsize]  {$f$};
\draw (200.61,248.9) node [anchor=north west][inner sep=0.75pt]  [font=\scriptsize]  {$\varphi $};
\draw (197.61,202.9) node [anchor=north west][inner sep=0.75pt]  [font=\scriptsize]  {$U$};

\end{tikzpicture}

\end{center}
there exists a red boundary square
\begin{center}

\tikzset{every picture/.style={line width=0.75pt}} 

\begin{tikzpicture}[x=0.75pt,y=0.75pt,yscale=-1,xscale=1]

\draw [color={rgb, 255:red, 208; green, 2; blue, 27 }  ,draw opacity=1 ]   (290,89.6) -- (340,89.6) ;
\draw [color={rgb, 255:red, 208; green, 2; blue, 27 }  ,draw opacity=1 ]   (290,139.6) -- (340,139.6) ;
\draw [color={rgb, 255:red, 208; green, 2; blue, 27 }  ,draw opacity=1 ]   (290,89.6) -- (290,139.6) ;
\draw [color={rgb, 255:red, 208; green, 2; blue, 27 }  ,draw opacity=1 ]   (340,89.6) -- (340,139.6) ;

\draw (304.61,107.9) node [anchor=north west][inner sep=0.75pt]  [font=\scriptsize]  {$f^{\ast } \varphi $};

\end{tikzpicture}

\end{center}
fitting in an equation of the form
\begin{center}

\tikzset{every picture/.style={line width=0.75pt}} 

\begin{tikzpicture}[x=0.75pt,y=0.75pt,yscale=-1,xscale=1]

\draw [color={rgb, 255:red, 208; green, 2; blue, 27 }  ,draw opacity=1 ]   (270,69.6) -- (320,69.6) ;
\draw [color={rgb, 255:red, 208; green, 2; blue, 27 }  ,draw opacity=1 ]   (270,119.6) -- (320,119.6) ;
\draw [color={rgb, 255:red, 208; green, 2; blue, 27 }  ,draw opacity=1 ]   (270,69.6) -- (270,119.6) ;
\draw [color={rgb, 255:red, 208; green, 2; blue, 27 }  ,draw opacity=1 ]   (320,69.6) -- (320,119.6) ;
\draw [color={rgb, 255:red, 208; green, 2; blue, 27 }  ,draw opacity=1 ]   (270,169.6) -- (320,169.6) ;
\draw [color={rgb, 255:red, 0; green, 0; blue, 0 }  ,draw opacity=1 ]   (270,119.6) -- (270,169.6) ;
\draw [color={rgb, 255:red, 0; green, 0; blue, 0 }  ,draw opacity=1 ]   (320,119.6) -- (320,169.6) ;
\draw [color={rgb, 255:red, 208; green, 2; blue, 27 }  ,draw opacity=1 ]   (171,69) -- (221,69) ;
\draw [color={rgb, 255:red, 208; green, 2; blue, 27 }  ,draw opacity=1 ]   (171,119) -- (221,119) ;
\draw    (171,69) -- (171,119) ;
\draw    (221,69) -- (221,119) ;
\draw [color={rgb, 255:red, 208; green, 2; blue, 27 }  ,draw opacity=1 ]   (171,169) -- (221,169) ;
\draw [color={rgb, 255:red, 208; green, 2; blue, 27 }  ,draw opacity=1 ]   (171,119) -- (171,169) ;
\draw [color={rgb, 255:red, 208; green, 2; blue, 27 }  ,draw opacity=1 ]   (221,119) -- (221,169) ;

\draw (322.36,137.04) node [anchor=north west][inner sep=0.75pt]  [font=\scriptsize]  {$f$};
\draw (259.36,136.04) node [anchor=north west][inner sep=0.75pt]  [font=\scriptsize]  {$f$};
\draw (288.61,136.9) node [anchor=north west][inner sep=0.75pt]  [font=\scriptsize]  {$U$};
\draw (284.61,87.9) node [anchor=north west][inner sep=0.75pt]  [font=\scriptsize]  {$f^{\ast } \varphi $};
\draw (223.36,87.44) node [anchor=north west][inner sep=0.75pt]  [font=\scriptsize]  {$f$};
\draw (160.36,87.44) node [anchor=north west][inner sep=0.75pt]  [font=\scriptsize]  {$f$};
\draw (191.61,136.3) node [anchor=north west][inner sep=0.75pt]  [font=\scriptsize]  {$\varphi $};
\draw (188.61,90.3) node [anchor=north west][inner sep=0.75pt]  [font=\scriptsize]  {$U$};
\draw (236.56,113.84) node [anchor=north west][inner sep=0.75pt]  [font=\scriptsize]  {$=$};

\end{tikzpicture}

\end{center}
Moreover, assume that the above squares can be chosen in such a way that for every composite vertical morphism $gf$, the chosen square $(gf)^\ast\varphi$ is the sliding $g^\ast f^\ast\varphi$ of the square $f^\ast \varphi$ along $g^\ast$. Geometrically
\begin{center}

\tikzset{every picture/.style={line width=0.75pt}} 

\begin{tikzpicture}[x=0.75pt,y=0.75pt,yscale=-1,xscale=1]

\draw [color={rgb, 255:red, 208; green, 2; blue, 27 }  ,draw opacity=1 ]   (290,89.6) -- (340,89.6) ;
\draw [color={rgb, 255:red, 208; green, 2; blue, 27 }  ,draw opacity=1 ]   (290,139.6) -- (340,139.6) ;
\draw [color={rgb, 255:red, 208; green, 2; blue, 27 }  ,draw opacity=1 ]   (290,89.6) -- (290,139.6) ;
\draw [color={rgb, 255:red, 208; green, 2; blue, 27 }  ,draw opacity=1 ]   (340,89.6) -- (340,139.6) ;
\draw [color={rgb, 255:red, 208; green, 2; blue, 27 }  ,draw opacity=1 ]   (290,189.6) -- (340,189.6) ;
\draw [color={rgb, 255:red, 0; green, 0; blue, 0 }  ,draw opacity=1 ]   (290,139.6) -- (290,189.6) ;
\draw [color={rgb, 255:red, 0; green, 0; blue, 0 }  ,draw opacity=1 ]   (340,139.6) -- (340,189.6) ;
\draw [color={rgb, 255:red, 208; green, 2; blue, 27 }  ,draw opacity=1 ]   (191,89) -- (241,89) ;
\draw [color={rgb, 255:red, 208; green, 2; blue, 27 }  ,draw opacity=1 ]   (191,139) -- (241,139) ;
\draw    (191,89) -- (191,139) ;
\draw    (241,89) -- (241,139) ;
\draw [color={rgb, 255:red, 208; green, 2; blue, 27 }  ,draw opacity=1 ]   (191,189) -- (241,189) ;
\draw [color={rgb, 255:red, 208; green, 2; blue, 27 }  ,draw opacity=1 ]   (191,139) -- (191,189) ;
\draw [color={rgb, 255:red, 208; green, 2; blue, 27 }  ,draw opacity=1 ]   (241,139) -- (241,189) ;
\draw [color={rgb, 255:red, 208; green, 2; blue, 27 }  ,draw opacity=1 ]   (191,39.6) -- (241,39.6) ;
\draw    (191,39.6) -- (191,89.6) ;
\draw    (241,39.6) -- (241,89.6) ;
\draw [color={rgb, 255:red, 208; green, 2; blue, 27 }  ,draw opacity=1 ]   (290.04,39.8) -- (340.04,39.8) ;
\draw    (290.04,39.8) -- (290.04,89.8) ;
\draw    (340.04,39.8) -- (340.04,89.8) ;
\draw [color={rgb, 255:red, 208; green, 2; blue, 27 }  ,draw opacity=1 ]   (389.5,90.1) -- (439.5,90.1) ;
\draw [color={rgb, 255:red, 208; green, 2; blue, 27 }  ,draw opacity=1 ]   (389.5,140.1) -- (439.5,140.1) ;
\draw [color={rgb, 255:red, 0; green, 0; blue, 0 }  ,draw opacity=1 ]   (389.5,90.1) -- (389.5,140.1) ;
\draw [color={rgb, 255:red, 0; green, 0; blue, 0 }  ,draw opacity=1 ]   (439.5,90.1) -- (439.5,140.1) ;
\draw [color={rgb, 255:red, 208; green, 2; blue, 27 }  ,draw opacity=1 ]   (389.5,190.1) -- (439.5,190.1) ;
\draw [color={rgb, 255:red, 0; green, 0; blue, 0 }  ,draw opacity=1 ]   (389.5,140.1) -- (389.5,190.1) ;
\draw [color={rgb, 255:red, 0; green, 0; blue, 0 }  ,draw opacity=1 ]   (439.5,140.1) -- (439.5,190.1) ;
\draw [color={rgb, 255:red, 208; green, 2; blue, 27 }  ,draw opacity=1 ]   (389.54,40.3) -- (439.54,40.3) ;
\draw [color={rgb, 255:red, 208; green, 2; blue, 27 }  ,draw opacity=1 ]   (389.54,40.3) -- (389.54,90.3) ;
\draw [color={rgb, 255:red, 208; green, 2; blue, 27 }  ,draw opacity=1 ]   (439.54,40.3) -- (439.54,90.3) ;

\draw (342.36,157.04) node [anchor=north west][inner sep=0.75pt]  [font=\scriptsize]  {$f$};
\draw (279.36,156.04) node [anchor=north west][inner sep=0.75pt]  [font=\scriptsize]  {$f$};
\draw (308.61,156.9) node [anchor=north west][inner sep=0.75pt]  [font=\scriptsize]  {$U$};
\draw (304.61,107.9) node [anchor=north west][inner sep=0.75pt]  [font=\scriptsize]  {$f^{\ast } \varphi $};
\draw (243.36,107.44) node [anchor=north west][inner sep=0.75pt]  [font=\scriptsize]  {$f$};
\draw (180.36,107.44) node [anchor=north west][inner sep=0.75pt]  [font=\scriptsize]  {$f$};
\draw (211.61,156.3) node [anchor=north west][inner sep=0.75pt]  [font=\scriptsize]  {$\varphi $};
\draw (208.61,110.3) node [anchor=north west][inner sep=0.75pt]  [font=\scriptsize]  {$U$};
\draw (256.56,103.34) node [anchor=north west][inner sep=0.75pt]  [font=\scriptsize]  {$=$};
\draw (243.36,58.04) node [anchor=north west][inner sep=0.75pt]  [font=\scriptsize]  {$g$};
\draw (180.36,58.04) node [anchor=north west][inner sep=0.75pt]  [font=\scriptsize]  {$g$};
\draw (208.61,60.9) node [anchor=north west][inner sep=0.75pt]  [font=\scriptsize]  {$U$};
\draw (342.4,58.24) node [anchor=north west][inner sep=0.75pt]  [font=\scriptsize]  {$g$};
\draw (279.4,58.24) node [anchor=north west][inner sep=0.75pt]  [font=\scriptsize]  {$g$};
\draw (307.65,61.1) node [anchor=north west][inner sep=0.75pt]  [font=\scriptsize]  {$U$};
\draw (441.86,157.54) node [anchor=north west][inner sep=0.75pt]  [font=\scriptsize]  {$f$};
\draw (378.86,156.54) node [anchor=north west][inner sep=0.75pt]  [font=\scriptsize]  {$f$};
\draw (408.11,157.4) node [anchor=north west][inner sep=0.75pt]  [font=\scriptsize]  {$U$};
\draw (399.11,59.4) node [anchor=north west][inner sep=0.75pt]  [font=\scriptsize]  {$g^{\ast } f^{\ast } \varphi $};
\draw (441.9,111.74) node [anchor=north west][inner sep=0.75pt]  [font=\scriptsize]  {$g$};
\draw (378.9,112.74) node [anchor=north west][inner sep=0.75pt]  [font=\scriptsize]  {$g$};
\draw (408.9,110.35) node [anchor=north west][inner sep=0.75pt]  [font=\scriptsize]  {$U$};
\draw (356.06,103.09) node [anchor=north west][inner sep=0.75pt]  [font=\scriptsize]  {$=$};

\end{tikzpicture}

\end{center}
in that case, the operation $\varphi\mapsto f^\ast \varphi$ defines a $\pi_2$-opindexing $\Phi$ of $H^\ast C$, and we will say that the double category $C$ induces the $\pi_2$-opindexing $\Phi$. More precisely.

\begin{definition}
  Let $\decB$ be a decorated bicategory. Let $\Phi$ be a $\pi_2$-opindexing on $\decB$. Let $C$ be an internalization of $\decB$. We will say that $C$ induces $\Phi$ if for every vertical morphism $f$ from object $a$ to object $b$ and for every $\varphi\in\pi_2(C,b)$ we have
 \begin{center}

\tikzset{every picture/.style={line width=0.75pt}} 

\begin{tikzpicture}[x=0.75pt,y=0.75pt,yscale=-1,xscale=1]

\draw [color={rgb, 255:red, 208; green, 2; blue, 27 }  ,draw opacity=1 ]   (160,161.6) -- (210,161.6) ;
\draw [color={rgb, 255:red, 208; green, 2; blue, 27 }  ,draw opacity=1 ]   (160,211.6) -- (210,211.6) ;
\draw    (160,161.6) -- (160,211.6) ;
\draw    (210,161.6) -- (210,211.6) ;
\draw [color={rgb, 255:red, 208; green, 2; blue, 27 }  ,draw opacity=1 ]   (160,261.6) -- (210,261.6) ;
\draw [color={rgb, 255:red, 208; green, 2; blue, 27 }  ,draw opacity=1 ]   (160,211.6) -- (160,261.6) ;
\draw [color={rgb, 255:red, 208; green, 2; blue, 27 }  ,draw opacity=1 ]   (210,211.6) -- (210,261.6) ;
\draw [color={rgb, 255:red, 208; green, 2; blue, 27 }  ,draw opacity=1 ]   (260,160.6) -- (310,160.6) ;
\draw [color={rgb, 255:red, 208; green, 2; blue, 27 }  ,draw opacity=1 ]   (260,210.6) -- (310,210.6) ;
\draw [color={rgb, 255:red, 208; green, 2; blue, 27 }  ,draw opacity=1 ]   (260,160.6) -- (260,210.6) ;
\draw [color={rgb, 255:red, 208; green, 2; blue, 27 }  ,draw opacity=1 ]   (310,160.6) -- (310,210.6) ;
\draw [color={rgb, 255:red, 208; green, 2; blue, 27 }  ,draw opacity=1 ]   (260,260.6) -- (310,260.6) ;
\draw [color={rgb, 255:red, 0; green, 0; blue, 0 }  ,draw opacity=1 ]   (260,210.6) -- (260,260.6) ;
\draw [color={rgb, 255:red, 0; green, 0; blue, 0 }  ,draw opacity=1 ]   (310,210.6) -- (310,260.6) ;

\draw (212.36,180.04) node [anchor=north west][inner sep=0.75pt]  [font=\scriptsize]  {$f$};
\draw (149.36,180.04) node [anchor=north west][inner sep=0.75pt]  [font=\scriptsize]  {$f$};
\draw (180.61,228.9) node [anchor=north west][inner sep=0.75pt]  [font=\scriptsize]  {$\varphi $};
\draw (177.61,182.9) node [anchor=north west][inner sep=0.75pt]  [font=\scriptsize]  {$U$};
\draw (312.36,228.04) node [anchor=north west][inner sep=0.75pt]  [font=\scriptsize]  {$f$};
\draw (249.36,227.04) node [anchor=north west][inner sep=0.75pt]  [font=\scriptsize]  {$f$};
\draw (278.61,227.9) node [anchor=north west][inner sep=0.75pt]  [font=\scriptsize]  {$U$};
\draw (269.61,181.9) node [anchor=north west][inner sep=0.75pt]  [font=\scriptsize]  {$\Phi _{f}( \varphi )$};
\draw (227.56,204.44) node [anchor=north west][inner sep=0.75pt]  [font=\scriptsize]  {$=$};

\end{tikzpicture}

 \end{center}
\noindent An analogous condition defines when an internalization induces a $\pi_2$-indexing.
\end{definition}
\noindent Observe that a double category $C$ induces a $\pi_2$-opindexing in the sense of Definition \ref{def:categoryinducing} if and only if $C$ admits an operation of sliding squares up unit squares, in the sense above. We can interpret the first part of the statement of Theorem \ref{thm:mainindexings} as saying that every $\pi_2$-opindexing on a decorated bicategory $\decB$ is induced by an internalization of $\decB$. Indeed, if $\Phi$ is a $\pi_2$-opindexing, then $\Phi$ is induced by $\crossb$.
\begin{definition}\label{def:categoryinducing}
Let $\decB$ be a decorated bicategory. Let $\Phi$ be a $\pi_2$-opindexing on $\decB$. Let $C$ and $C'$ be double categories inducing $\Phi$. A morphism from $C$ to $C'$ is a strict double functor $F:C\to C'$ such that the assignment associating to every object $a$ of $B^\ast$, the restriction $F_1:\pi_2(C,a)\to \pi_2(C,F_0(a))$ defines a natural transformation
\begin{center}
    \begin{tikzpicture}
\matrix(m)[matrix of math nodes, row sep=2em, column sep=2em,text height=2ex, text depth=0.25ex]
{B^\ast &&\\
&& \textbf{commMon}\\
B^\ast &&\\};
\path[->,font=\scriptsize,>=angle 90]
(m-1-1) edge node [left]{$F_0$} (m-3-1)
        edge node [above]{$\Phi$} (m-2-3)
        edge [white,bend left=55] node [black]{$F_1\Downarrow$}(m-3-1)
(m-3-1) edge node [below]{$\Phi$} (m-2-3)

;
\end{tikzpicture}
\end{center}
\noindent We will write $\dcat _\Phi$ for the category whose objects are internalizations of $\decB$ inducing $\Phi$ and whose morphisms are morphisms as above.
\end{definition}
\begin{obs}\label{obs:morphism}
Let $\decB$. Let $\Phi$ be a $\pi_2$-opindexing on $\decB$. Let $C$ and $C'$ be internalizations of $\decB$, inducing $\Phi$. The condition appearing in Definition \ref{def:categoryinducing} says that applying the square component $F_1$ of the double functor $F_1$ to a square of the form
\begin{center}

\tikzset{every picture/.style={line width=0.75pt}} 

\begin{tikzpicture}[x=0.75pt,y=0.75pt,yscale=-1,xscale=1]

\draw [color={rgb, 255:red, 208; green, 2; blue, 27 }  ,draw opacity=1 ]   (180,202.6) -- (230,202.6) ;
\draw [color={rgb, 255:red, 208; green, 2; blue, 27 }  ,draw opacity=1 ]   (180,252.6) -- (230,252.6) ;
\draw [color={rgb, 255:red, 208; green, 2; blue, 27 }  ,draw opacity=1 ]   (180,202.6) -- (180,252.6) ;
\draw [color={rgb, 255:red, 208; green, 2; blue, 27 }  ,draw opacity=1 ]   (230,202.6) -- (230,252.6) ;

\draw (194.61,219.9) node [anchor=north west][inner sep=0.75pt]  [font=\scriptsize]  {$\Phi _{f} \varphi $};

\end{tikzpicture}

\end{center}
one obtains the equation 
\begin{center}

\tikzset{every picture/.style={line width=0.75pt}} 

\begin{tikzpicture}[x=0.75pt,y=0.75pt,yscale=-1,xscale=1]

\draw [color={rgb, 255:red, 208; green, 2; blue, 27 }  ,draw opacity=1 ]   (200,131) -- (250,131) ;
\draw [color={rgb, 255:red, 208; green, 2; blue, 27 }  ,draw opacity=1 ]   (200,181) -- (250,181) ;
\draw [color={rgb, 255:red, 208; green, 2; blue, 27 }  ,draw opacity=1 ]   (200,131) -- (200,181) ;
\draw [color={rgb, 255:red, 208; green, 2; blue, 27 }  ,draw opacity=1 ]   (250,131) -- (250,181) ;
\draw [color={rgb, 255:red, 208; green, 2; blue, 27 }  ,draw opacity=1 ]   (291,130) -- (341,130) ;
\draw [color={rgb, 255:red, 208; green, 2; blue, 27 }  ,draw opacity=1 ]   (291,180) -- (341,180) ;
\draw [color={rgb, 255:red, 208; green, 2; blue, 27 }  ,draw opacity=1 ]   (291,130) -- (291,180) ;
\draw [color={rgb, 255:red, 208; green, 2; blue, 27 }  ,draw opacity=1 ]   (341,130) -- (341,180) ;

\draw (206.61,148.3) node [anchor=north west][inner sep=0.75pt]  [font=\scriptsize]  {$F_{1} \Phi _{f} \varphi $};
\draw (299.61,147.3) node [anchor=north west][inner sep=0.75pt]  [font=\scriptsize]  {$\Phi _{f_{0} f} \varphi $};
\draw (262.61,148.7) node [anchor=north west][inner sep=0.75pt]  [font=\scriptsize]  {$=$};

\end{tikzpicture}

\end{center}
i.e. applying $F_1$ to an equation of the form
\begin{center}

\tikzset{every picture/.style={line width=0.75pt}} 

\begin{tikzpicture}[x=0.75pt,y=0.75pt,yscale=-1,xscale=1]

\draw [color={rgb, 255:red, 208; green, 2; blue, 27 }  ,draw opacity=1 ]   (160,161.6) -- (210,161.6) ;
\draw [color={rgb, 255:red, 208; green, 2; blue, 27 }  ,draw opacity=1 ]   (160,211.6) -- (210,211.6) ;
\draw    (160,161.6) -- (160,211.6) ;
\draw    (210,161.6) -- (210,211.6) ;
\draw [color={rgb, 255:red, 208; green, 2; blue, 27 }  ,draw opacity=1 ]   (160,261.6) -- (210,261.6) ;
\draw [color={rgb, 255:red, 208; green, 2; blue, 27 }  ,draw opacity=1 ]   (160,211.6) -- (160,261.6) ;
\draw [color={rgb, 255:red, 208; green, 2; blue, 27 }  ,draw opacity=1 ]   (210,211.6) -- (210,261.6) ;
\draw [color={rgb, 255:red, 208; green, 2; blue, 27 }  ,draw opacity=1 ]   (260,160.6) -- (310,160.6) ;
\draw [color={rgb, 255:red, 208; green, 2; blue, 27 }  ,draw opacity=1 ]   (260,210.6) -- (310,210.6) ;
\draw [color={rgb, 255:red, 208; green, 2; blue, 27 }  ,draw opacity=1 ]   (260,160.6) -- (260,210.6) ;
\draw [color={rgb, 255:red, 208; green, 2; blue, 27 }  ,draw opacity=1 ]   (310,160.6) -- (310,210.6) ;
\draw [color={rgb, 255:red, 208; green, 2; blue, 27 }  ,draw opacity=1 ]   (260,260.6) -- (310,260.6) ;
\draw [color={rgb, 255:red, 0; green, 0; blue, 0 }  ,draw opacity=1 ]   (260,210.6) -- (260,260.6) ;
\draw [color={rgb, 255:red, 0; green, 0; blue, 0 }  ,draw opacity=1 ]   (310,210.6) -- (310,260.6) ;

\draw (212.36,180.04) node [anchor=north west][inner sep=0.75pt]  [font=\scriptsize]  {$f$};
\draw (149.36,180.04) node [anchor=north west][inner sep=0.75pt]  [font=\scriptsize]  {$f$};
\draw (180.61,228.9) node [anchor=north west][inner sep=0.75pt]  [font=\scriptsize]  {$\varphi $};
\draw (177.61,182.9) node [anchor=north west][inner sep=0.75pt]  [font=\scriptsize]  {$U$};
\draw (312.36,228.04) node [anchor=north west][inner sep=0.75pt]  [font=\scriptsize]  {$f$};
\draw (249.36,227.04) node [anchor=north west][inner sep=0.75pt]  [font=\scriptsize]  {$f$};
\draw (278.61,227.9) node [anchor=north west][inner sep=0.75pt]  [font=\scriptsize]  {$U$};
\draw (269.61,181.9) node [anchor=north west][inner sep=0.75pt]  [font=\scriptsize]  {$\Phi _{f}( \varphi )$};
\draw (227.56,204.44) node [anchor=north west][inner sep=0.75pt]  [font=\scriptsize]  {$=$};

\end{tikzpicture}

\end{center}
one obtains the equation
\begin{center}

\tikzset{every picture/.style={line width=0.75pt}} 

\begin{tikzpicture}[x=0.75pt,y=0.75pt,yscale=-1,xscale=1]

\draw [color={rgb, 255:red, 208; green, 2; blue, 27 }  ,draw opacity=1 ]   (230,90.6) -- (280,90.6) ;
\draw [color={rgb, 255:red, 208; green, 2; blue, 27 }  ,draw opacity=1 ]   (230,140.6) -- (280,140.6) ;
\draw    (230,90.6) -- (230,140.6) ;
\draw    (280,90.6) -- (280,140.6) ;
\draw [color={rgb, 255:red, 208; green, 2; blue, 27 }  ,draw opacity=1 ]   (230,190.6) -- (280,190.6) ;
\draw [color={rgb, 255:red, 208; green, 2; blue, 27 }  ,draw opacity=1 ]   (230,140.6) -- (230,190.6) ;
\draw [color={rgb, 255:red, 208; green, 2; blue, 27 }  ,draw opacity=1 ]   (280,140.6) -- (280,190.6) ;
\draw [color={rgb, 255:red, 208; green, 2; blue, 27 }  ,draw opacity=1 ]   (330,89.6) -- (380,89.6) ;
\draw [color={rgb, 255:red, 208; green, 2; blue, 27 }  ,draw opacity=1 ]   (330,139.6) -- (380,139.6) ;
\draw [color={rgb, 255:red, 208; green, 2; blue, 27 }  ,draw opacity=1 ]   (330,89.6) -- (330,139.6) ;
\draw [color={rgb, 255:red, 208; green, 2; blue, 27 }  ,draw opacity=1 ]   (380,89.6) -- (380,139.6) ;
\draw [color={rgb, 255:red, 208; green, 2; blue, 27 }  ,draw opacity=1 ]   (330,189.6) -- (380,189.6) ;
\draw [color={rgb, 255:red, 0; green, 0; blue, 0 }  ,draw opacity=1 ]   (330,139.6) -- (330,189.6) ;
\draw [color={rgb, 255:red, 0; green, 0; blue, 0 }  ,draw opacity=1 ]   (380,139.6) -- (380,189.6) ;

\draw (281.36,109.04) node [anchor=north west][inner sep=0.75pt]  [font=\scriptsize]  {$\ Ff$};
\draw (210.36,109.04) node [anchor=north west][inner sep=0.75pt]  [font=\scriptsize]  {$Ff$};
\draw (246.61,157.9) node [anchor=north west][inner sep=0.75pt]  [font=\scriptsize]  {$F\varphi $};
\draw (247.61,111.9) node [anchor=north west][inner sep=0.75pt]  [font=\scriptsize]  {$U$};
\draw (384.36,157.04) node [anchor=north west][inner sep=0.75pt]  [font=\scriptsize]  {$Ff$};
\draw (311.36,156.04) node [anchor=north west][inner sep=0.75pt]  [font=\scriptsize]  {$Ff$};
\draw (348.61,156.9) node [anchor=north west][inner sep=0.75pt]  [font=\scriptsize]  {$U$};
\draw (330.61,110.9) node [anchor=north west][inner sep=0.75pt]  [font=\scriptsize]  {$\Phi _{Ff}( F\varphi )$};
\draw (297.56,133.44) node [anchor=north west][inner sep=0.75pt]  [font=\scriptsize]  {$=$};

\end{tikzpicture}

\end{center}
Symbolically this is expressed in the following simple equation for every red boundary square $\varphi$ and every appropriate vertical morphism $f$
\begin{equation}\label{eq:morphismequation}
 F_1(\Phi_f(\varphi))=\Phi_{F_0f}(F_1\varphi)   
\end{equation}
\end{obs}
\noindent Although double categories inducing a $\pi_2$-opindexing $\Phi$ on a decorated internalization $\decB$ more naturally organize into the category $\dcat_\Phi$, we are more interested in the subcategory $\overline{\dcat}_\Phi$ of $\dcat_\Phi$ having the same objects as $\dcat_\Phi$, but where morphisms are now morphisms $F:C\to C'$ in $\dcat_\Phi$ such that $H^\ast F=id_{\decB}$, i.e. such that $F$ is the identity on globular squares and vertical morphisms. Such morphisms have already been considered in e.g. \cite{Orendain3}. The next theorem is the main result of this note. It says that the construction of the double category $\crossb$ from the data of a $\pi_2$-opindexing $\Phi$ is universal.
\begin{thm}\label{thm:initial}
Let $\decB$ be a decorated bicategory. Let $\Phi$ be a $\pi_2$-indexing on $\decB$. The double category $\crossb$ is an initial object in $\overline{\dcat_\Phi}$. Moreover, given a double category inducing $\Phi$, the unique morphism
\[!:\crossb\to C\]
is such that the decorated horzontalization $H^\ast !=id_{\decB}$ and $!$ is full on globularly generated squares of $C$.
\end{thm}

\begin{proof}
Squares in $\crossb$ are either globular squares or are of the form 
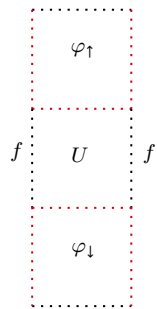
\begin{center}

\tikzset{every picture/.style={line width=0.75pt}} 



\end{center}
In symbols, nonglobular squares in $\crossb$ are of the form $(\varphi_\downarrow, f,\varphi_\uparrow)$ for globular squares $\varphi_\uparrow,\varphi_\downarrow$, and a nontrivial vertical morphism $f$. The condition $H^\ast !=id_{\decB}$ determines $!$ on objects, arrows, and globular squares of $\crossb$. The assignment $!$ is trivially compatible with horizontal and vertical compositions of globular squares. Let $\varphi$ be
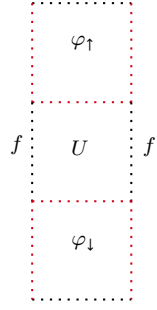
\begin{center}

\tikzset{every picture/.style={line width=0.75pt}} 

%

\end{center}
be a nonglobular square in $\crossb$. Again, the condition $H^\ast !=id_{\decB}$, together with the condition that $!$ should be a strict double functor, forces $!$ applied to $\varphi$ to be the composite square
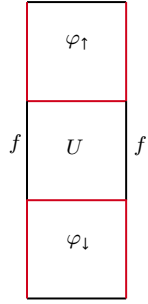
\begin{center}

\tikzset{every picture/.style={line width=0.75pt}} 

%

\end{center}
\noindent i.e. $!$ must be the assignment turning totted squares in $\crossb$ into solid squares in $C$. We must prove that this assignment is well defined and is compatible with identity squares, and compositions. We first prove that $!$ is well defined. Let
  
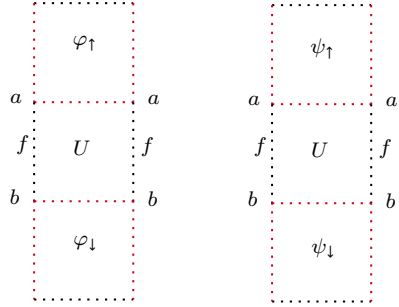
\begin{center}

\tikzset{every picture/.style={line width=0.75pt}} 

%

    \end{center}
\noindent be two canonical form presentations of the same square in $\crossb$. By observation $\ref{obs:squarescrosspsame}$ there exists a dotted red boundary square
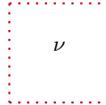
\begin{center}

\tikzset{every picture/.style={line width=0.75pt}} 

%

    \end{center}
    \noindent in $\pi_2(C,b)$ such that the following two equations hold
    \begin{center}

\tikzset{every picture/.style={line width=0.75pt}} 

%

    \end{center}
We must prove that the squares
\begin{center}

\tikzset{every picture/.style={line width=0.75pt}} 

%

\end{center}
in $C$ are equal. The following equations
\begin{center}

\tikzset{every picture/.style={line width=0.75pt}} 

%

\end{center}
follows from its dotted version, as both equations are equations in $\decB$. The equality we wish to prove follows from this and from the following equation
\begin{center}

\tikzset{every picture/.style={line width=0.75pt}} 

%

\end{center}
We now prove that $!$ is compatible with identity squares and composition of squares in $\crossb$ and $C$. Let $f$ be a nontrivial vertical morphism. The horizontal identity square of $f$ in $\decB$ is the square $(id, f, id)$, i.e. is equal to
\begin{center}

\tikzset{every picture/.style={line width=0.75pt}} 

%

\end{center}
$!$ applied to this square is equal to the square
\begin{center}

\tikzset{every picture/.style={line width=0.75pt}} 

%

\end{center}
which is equal to the unit square $U_f$ of $f$ in $C$. $!$ is thus compatible with horizontal identity squares. Given two horizontally composable squares $\varphi=(\varphi_\downarrow, f,\varphi_\uparrow)$ and $\psi=(\psi_\downarrow, f,\psi_\uparrow)$, i.e.
\begin{center}

\tikzset{every picture/.style={line width=0.75pt}} 

%

\end{center}
the horizontal composition $\varphi\boxvert\psi$ is the square
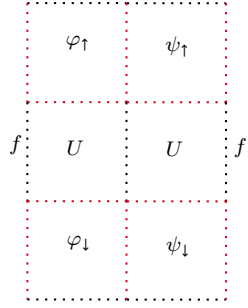
\begin{center}

\tikzset{every picture/.style={line width=0.75pt}} 

%

\end{center}
where the $\boxvert$ symbol inside the brackets is the horizontal composition of the corresponding globular squares in $HC$. $!(\varphi\boxvert \psi)$ is thus the horizontal composition
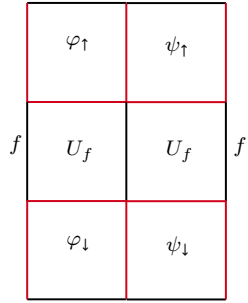
\begin{center}

\tikzset{every picture/.style={line width=0.75pt}} 

%

\end{center}
in $C$. By the exchange relation, this square can be interpreted as the vertical composite of the corresponding three horizontal composite squares above. The middle square is the horizontal identity square $U_f$, and thus the above square is
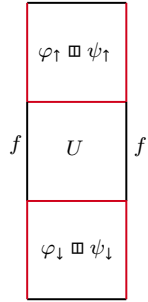
\begin{center}

\tikzset{every picture/.style={line width=0.75pt}} 

%

\end{center}
The image $!(\varphi\boxvert \psi)$ of $\varphi\boxvert\psi$, under $!$, is thus $!(\varphi)\boxvert !(\psi)$. We conclude that $!$ is compatible with horizontal composition. $!$ is clearly compatible with vertical identities. We now prove compatibility with vertical composition. Let $\varphi=(\varphi_\downarrow, f,\varphi_\uparrow)$ and $\psi=(\psi_\downarrow, g,\psi_\uparrow)$, where $\varphi_\downarrow$ and $\psi_\uparrow$ are vertically composable globular squares in $C$. The vertical composition $\psi\boxminus \varphi$ is the square 

\[(\psi_\downarrow\boxminus\Phi_g(\psi_\uparrow\boxminus\varphi_\downarrow), gf, \varphi_\uparrow)\]
Geometrically
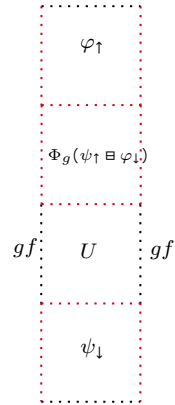
\begin{center}

\tikzset{every picture/.style={line width=0.75pt}} 

\begin{tikzpicture}[x=0.75pt,y=0.75pt,yscale=-1,xscale=1]

\draw [color={rgb, 255:red, 0; green, 0; blue, 0 }  ,draw opacity=1 ] [dash pattern={on 0.84pt off 2.51pt}]  (302.1,28.16) -- (352.1,28.16) ;
\draw [color={rgb, 255:red, 208; green, 2; blue, 27 }  ,draw opacity=1 ] [dash pattern={on 0.84pt off 2.51pt}]  (302.1,28.16) -- (302.1,78.16) ;
\draw [color={rgb, 255:red, 208; green, 2; blue, 27 }  ,draw opacity=1 ] [dash pattern={on 0.84pt off 2.51pt}]  (352.1,28.16) -- (352.1,78.16) ;
\draw [color={rgb, 255:red, 0; green, 0; blue, 0 }  ,draw opacity=1 ] [dash pattern={on 0.84pt off 2.51pt}]  (302,128.21) -- (302,178.21) ;
\draw [color={rgb, 255:red, 0; green, 0; blue, 0 }  ,draw opacity=1 ] [dash pattern={on 0.84pt off 2.51pt}]  (352,128.21) -- (352,178.21) ;
\draw [color={rgb, 255:red, 208; green, 2; blue, 27 }  ,draw opacity=1 ] [dash pattern={on 0.84pt off 2.51pt}]  (302,178.21) -- (352,178.21) ;
\draw [color={rgb, 255:red, 0; green, 0; blue, 0 }  ,draw opacity=1 ] [dash pattern={on 0.84pt off 2.51pt}]  (302,227.81) -- (352,227.81) ;
\draw [color={rgb, 255:red, 208; green, 2; blue, 27 }  ,draw opacity=1 ] [dash pattern={on 0.84pt off 2.51pt}]  (302,177.81) -- (302,227.81) ;
\draw [color={rgb, 255:red, 208; green, 2; blue, 27 }  ,draw opacity=1 ] [dash pattern={on 0.84pt off 2.51pt}]  (352,177.81) -- (352,227.81) ;
\draw [color={rgb, 255:red, 208; green, 2; blue, 27 }  ,draw opacity=1 ] [dash pattern={on 0.84pt off 2.51pt}]  (302.17,128.07) -- (352.17,128.07) ;
\draw [color={rgb, 255:red, 208; green, 2; blue, 27 }  ,draw opacity=1 ] [dash pattern={on 0.84pt off 2.51pt}]  (302.17,78.07) -- (302.17,128.07) ;
\draw [color={rgb, 255:red, 208; green, 2; blue, 27 }  ,draw opacity=1 ] [dash pattern={on 0.84pt off 2.51pt}]  (352.17,78.07) -- (352.17,128.07) ;
\draw [color={rgb, 255:red, 208; green, 2; blue, 27 }  ,draw opacity=1 ] [dash pattern={on 0.84pt off 2.51pt}]  (302.17,78.07) -- (352.17,78.07) ;

\draw (320.06,43.4) node [anchor=north west][inner sep=0.75pt]  [font=\scriptsize]  {$\varphi _{\uparrow }$};
\draw (355.36,144.65) node [anchor=north west][inner sep=0.75pt]  [font=\scriptsize]  {$gf$};
\draw (286.36,143.65) node [anchor=north west][inner sep=0.75pt]  [font=\scriptsize]  {$gf$};
\draw (320.36,194.25) node [anchor=north west][inner sep=0.75pt]  [font=\scriptsize]  {$\psi _{\downarrow }$};
\draw (320.21,146.31) node [anchor=north west][inner sep=0.75pt]  [font=\scriptsize]  {$U$};
\draw (303.8,97.98) node [anchor=north west][inner sep=0.75pt]  [font=\tiny]  {$\Phi _{g}( \psi _{\uparrow } \boxminus \varphi _{\downarrow })$};

\end{tikzpicture}

\end{center}
The vertical composition $!(\psi)\boxminus !(\psi)$ is, on the other hand, the composite
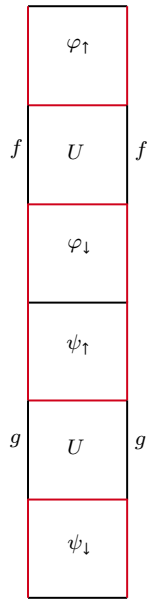
\begin{center}

\tikzset{every picture/.style={line width=0.75pt}} 

\begin{tikzpicture}[x=0.75pt,y=0.75pt,yscale=-1,xscale=1]

\draw [color={rgb, 255:red, 0; green, 0; blue, 0 }  ,draw opacity=1 ]   (300,69.4) -- (350,69.4) ;
\draw [color={rgb, 255:red, 208; green, 2; blue, 27 }  ,draw opacity=1 ]   (300,119.4) -- (350,119.4) ;
\draw [color={rgb, 255:red, 208; green, 2; blue, 27 }  ,draw opacity=1 ]   (300,69.4) -- (300,119.4) ;
\draw [color={rgb, 255:red, 208; green, 2; blue, 27 }  ,draw opacity=1 ]   (350,69.4) -- (350,119.4) ;
\draw [color={rgb, 255:red, 0; green, 0; blue, 0 }  ,draw opacity=1 ]   (300,119.4) -- (300,169.4) ;
\draw [color={rgb, 255:red, 0; green, 0; blue, 0 }  ,draw opacity=1 ]   (350,119.4) -- (350,169.4) ;
\draw [color={rgb, 255:red, 208; green, 2; blue, 27 }  ,draw opacity=1 ]   (300,169.4) -- (350,169.4) ;
\draw [color={rgb, 255:red, 0; green, 0; blue, 0 }  ,draw opacity=1 ]   (300,219) -- (350,219) ;
\draw [color={rgb, 255:red, 208; green, 2; blue, 27 }  ,draw opacity=1 ]   (300,169) -- (300,219) ;
\draw [color={rgb, 255:red, 208; green, 2; blue, 27 }  ,draw opacity=1 ]   (350,169) -- (350,219) ;
\draw [color={rgb, 255:red, 208; green, 2; blue, 27 }  ,draw opacity=1 ]   (300,268.4) -- (350,268.4) ;
\draw [color={rgb, 255:red, 208; green, 2; blue, 27 }  ,draw opacity=1 ]   (300,218.4) -- (300,268.4) ;
\draw [color={rgb, 255:red, 208; green, 2; blue, 27 }  ,draw opacity=1 ]   (350,218.4) -- (350,268.4) ;
\draw [color={rgb, 255:red, 0; green, 0; blue, 0 }  ,draw opacity=1 ]   (300,268.4) -- (300,318.4) ;
\draw [color={rgb, 255:red, 0; green, 0; blue, 0 }  ,draw opacity=1 ]   (350,268.4) -- (350,318.4) ;
\draw [color={rgb, 255:red, 208; green, 2; blue, 27 }  ,draw opacity=1 ]   (300,318.4) -- (350,318.4) ;
\draw [color={rgb, 255:red, 0; green, 0; blue, 0 }  ,draw opacity=1 ]   (300,368) -- (350,368) ;
\draw [color={rgb, 255:red, 208; green, 2; blue, 27 }  ,draw opacity=1 ]   (300,318) -- (300,368) ;
\draw [color={rgb, 255:red, 208; green, 2; blue, 27 }  ,draw opacity=1 ]   (350,318) -- (350,368) ;

\draw (352.36,135.84) node [anchor=north west][inner sep=0.75pt]  [font=\scriptsize]  {$f$};
\draw (289.36,134.84) node [anchor=north west][inner sep=0.75pt]  [font=\scriptsize]  {$f$};
\draw (318.36,185.44) node [anchor=north west][inner sep=0.75pt]  [font=\scriptsize]  {$\varphi _{\downarrow }$};
\draw (318.21,137.5) node [anchor=north west][inner sep=0.75pt]  [font=\scriptsize]  {$U$};
\draw (317.96,84.64) node [anchor=north west][inner sep=0.75pt]  [font=\scriptsize]  {$\varphi _{\uparrow }$};
\draw (352.36,284.84) node [anchor=north west][inner sep=0.75pt]  [font=\scriptsize]  {$g$};
\draw (289.36,283.84) node [anchor=north west][inner sep=0.75pt]  [font=\scriptsize]  {$g$};
\draw (318.36,334.44) node [anchor=north west][inner sep=0.75pt]  [font=\scriptsize]  {$\psi _{\downarrow }$};
\draw (318.21,286.5) node [anchor=north west][inner sep=0.75pt]  [font=\scriptsize]  {$U$};
\draw (317.96,233.64) node [anchor=north west][inner sep=0.75pt]  [font=\scriptsize]  {$\psi _{\uparrow }$};

\end{tikzpicture}

\end{center}
Consider the middle square
\begin{center}

\tikzset{every picture/.style={line width=0.75pt}} 

\begin{tikzpicture}[x=0.75pt,y=0.75pt,yscale=-1,xscale=1]

\draw [color={rgb, 255:red, 0; green, 0; blue, 0 }  ,draw opacity=1 ]   (320,91) -- (320,141) ;
\draw [color={rgb, 255:red, 0; green, 0; blue, 0 }  ,draw opacity=1 ]   (370,91) -- (370,141) ;
\draw [color={rgb, 255:red, 208; green, 2; blue, 27 }  ,draw opacity=1 ]   (320,141) -- (370,141) ;
\draw [color={rgb, 255:red, 0; green, 0; blue, 0 }  ,draw opacity=1 ]   (320,190.6) -- (370,190.6) ;
\draw [color={rgb, 255:red, 208; green, 2; blue, 27 }  ,draw opacity=1 ]   (320,140.6) -- (320,190.6) ;
\draw [color={rgb, 255:red, 208; green, 2; blue, 27 }  ,draw opacity=1 ]   (370,140.6) -- (370,190.6) ;
\draw [color={rgb, 255:red, 208; green, 2; blue, 27 }  ,draw opacity=1 ]   (320,240) -- (370,240) ;
\draw [color={rgb, 255:red, 208; green, 2; blue, 27 }  ,draw opacity=1 ]   (320,190) -- (320,240) ;
\draw [color={rgb, 255:red, 208; green, 2; blue, 27 }  ,draw opacity=1 ]   (370,190) -- (370,240) ;
\draw [color={rgb, 255:red, 208; green, 2; blue, 27 }  ,draw opacity=1 ]   (320,91) -- (370,91) ;

\draw (372.36,107.44) node [anchor=north west][inner sep=0.75pt]  [font=\scriptsize]  {$g$};
\draw (309.36,106.44) node [anchor=north west][inner sep=0.75pt]  [font=\scriptsize]  {$g$};
\draw (338.36,157.04) node [anchor=north west][inner sep=0.75pt]  [font=\scriptsize]  {$\varphi _{\downarrow }$};
\draw (338.21,109.1) node [anchor=north west][inner sep=0.75pt]  [font=\scriptsize]  {$U$};
\draw (337.96,205.24) node [anchor=north west][inner sep=0.75pt]  [font=\scriptsize]  {$\psi _{\uparrow }$};

\end{tikzpicture}

\end{center}
By the condition that $C$ induces $\Phi$ his is equal, in $C$, to the square
\begin{center}

\tikzset{every picture/.style={line width=0.75pt}} 

\begin{tikzpicture}[x=0.75pt,y=0.75pt,yscale=-1,xscale=1]

\draw [color={rgb, 255:red, 0; green, 0; blue, 0 }  ,draw opacity=1 ]   (340,111) -- (340,161) ;
\draw [color={rgb, 255:red, 0; green, 0; blue, 0 }  ,draw opacity=1 ]   (390,111) -- (390,161) ;
\draw [color={rgb, 255:red, 208; green, 2; blue, 27 }  ,draw opacity=1 ]   (340,161) -- (390,161) ;
\draw [color={rgb, 255:red, 208; green, 2; blue, 27 }  ,draw opacity=1 ]   (340,111) -- (390,111) ;
\draw [color={rgb, 255:red, 208; green, 2; blue, 27 }  ,draw opacity=1 ]   (340,60.67) -- (340,110.67) ;
\draw [color={rgb, 255:red, 208; green, 2; blue, 27 }  ,draw opacity=1 ]   (390,60.67) -- (390,110.67) ;
\draw [color={rgb, 255:red, 208; green, 2; blue, 27 }  ,draw opacity=1 ]   (340,60.67) -- (390,60.67) ;

\draw (392.36,127.44) node [anchor=north west][inner sep=0.75pt]  [font=\scriptsize]  {$g$};
\draw (329.36,126.44) node [anchor=north west][inner sep=0.75pt]  [font=\scriptsize]  {$g$};
\draw (358.21,129.1) node [anchor=north west][inner sep=0.75pt]  [font=\scriptsize]  {$U$};
\draw (340.63,80.57) node [anchor=north west][inner sep=0.75pt]  [font=\tiny]  {$\Phi _{g}( \psi _{\uparrow } \boxminus \varphi _{\downarrow })$};

\end{tikzpicture}

\end{center}
Substituting, the composite $!(\psi)\boxminus !(\psi)$ is equal to the square
\begin{center}

\tikzset{every picture/.style={line width=0.75pt}} 

\begin{tikzpicture}[x=0.75pt,y=0.75pt,yscale=-1,xscale=1]

\draw [color={rgb, 255:red, 0; green, 0; blue, 0 }  ,draw opacity=1 ]   (282.1,201.35) -- (332.1,201.35) ;
\draw [color={rgb, 255:red, 208; green, 2; blue, 27 }  ,draw opacity=1 ]   (282.1,201.35) -- (282.1,251.35) ;
\draw [color={rgb, 255:red, 208; green, 2; blue, 27 }  ,draw opacity=1 ]   (332.1,201.35) -- (332.1,251.35) ;
\draw [color={rgb, 255:red, 0; green, 0; blue, 0 }  ,draw opacity=1 ]   (282,301.4) -- (282,351.4) ;
\draw [color={rgb, 255:red, 0; green, 0; blue, 0 }  ,draw opacity=1 ]   (332,301.4) -- (332,351.4) ;
\draw [color={rgb, 255:red, 208; green, 2; blue, 27 }  ,draw opacity=1 ]   (282,351.4) -- (332,351.4) ;
\draw [color={rgb, 255:red, 0; green, 0; blue, 0 }  ,draw opacity=1 ]   (282,401) -- (332,401) ;
\draw [color={rgb, 255:red, 208; green, 2; blue, 27 }  ,draw opacity=1 ]   (282,351) -- (282,401) ;
\draw [color={rgb, 255:red, 208; green, 2; blue, 27 }  ,draw opacity=1 ]   (332,351) -- (332,401) ;
\draw [color={rgb, 255:red, 208; green, 2; blue, 27 }  ,draw opacity=1 ]   (282.17,301.27) -- (332.17,301.27) ;
\draw [color={rgb, 255:red, 208; green, 2; blue, 27 }  ,draw opacity=1 ]   (282.17,251.27) -- (282.17,301.27) ;
\draw [color={rgb, 255:red, 208; green, 2; blue, 27 }  ,draw opacity=1 ]   (332.17,251.27) -- (332.17,301.27) ;
\draw [color={rgb, 255:red, 208; green, 2; blue, 27 }  ,draw opacity=1 ]   (282.17,251.27) -- (332.17,251.27) ;

\draw (300.06,216.59) node [anchor=north west][inner sep=0.75pt]  [font=\scriptsize]  {$\varphi _{\uparrow }$};
\draw (335.36,317.84) node [anchor=north west][inner sep=0.75pt]  [font=\scriptsize]  {$gf$};
\draw (266.36,316.84) node [anchor=north west][inner sep=0.75pt]  [font=\scriptsize]  {$gf$};
\draw (300.36,367.44) node [anchor=north west][inner sep=0.75pt]  [font=\scriptsize]  {$\psi _{\downarrow }$};
\draw (300.21,319.5) node [anchor=north west][inner sep=0.75pt]  [font=\scriptsize]  {$U$};
\draw (283.8,271.17) node [anchor=north west][inner sep=0.75pt]  [font=\tiny]  {$\Phi _{g}( \psi _{\uparrow } \boxminus \varphi _{\downarrow })$};

\end{tikzpicture}

\end{center}
This is precisely the square $!(\varphi\boxminus\psi)$. The double functor $!$ is thus compatible with vertical composition, and is thus a strict double functor $!:\crossb\to C$. Moreover, $!$ is clearly the unique double functor $\crossb\to C$ satisfying the equation $H^\ast q=id_{\decB}$. We conclude that $\crossb$ is an initial objec in $\overline{\dcat}_\Phi$. Finally, $!$ is clearly surjective on globular squares and horizontal identity squares of $C$. The fact that every square in $\gamma C$ admits a subdivision into globular and horizontal identity squares, implies that $!$ is surjective on squares. This concludes the proof.
\end{proof}
\noindent We interpret the conditions on the morphism $!$ in the statement of Theorem \ref{thm:initial} as the condition that the double category $\crossb$ parametrizes the globularly generated piece of every internalization of $\decB$ inducing the $\pi_2$-opindexing $\Phi$, and we interpret the unique double functor $!$ as a kind of evaluation, suggestively represented as making dotted squares in canonical form into solid squares. Thinking about the conditions satisfied by $!$ in this way makes it clear that every globularly generated square in $C$ has the same 'shape' as a corresponding square in $\crossb$. Every square in $\crossb$ is canonical and thus every square in $C$ is also canonical. From this and from \cite[Lemma 3.9]{OrendainMaldonado} we obtain the following.
\begin{cor}\label{cor:mainthm}
Let $\decB$ be a decorated bicategory. Let $\Phi$ be a $\pi_2$-opindexing/$\pi_2$-indexing on $\decB$. If $C$ is a double category inducing $\Phi$, then $\ell C=1$.
\end{cor}

\section{Fully faithful and absolutely dense framed bicategories}\label{s:fully faithful}
\noindent In this section we define the conditions of a framed bicategory being fully faithful and absolutely dense. Our main application of Theorem \ref{thm:initial} and Corollary \ref{cor:mainthm} will be to describe the structure of fully faithful and absolutely dense framed bicategories and in particular to prove that they are of length 1. Fully faithful/absolutely dense framed bicategories are framed bicategories for which every vertical arrow is fully faithful/absolutely dense, in the sense of \cite[Definition 4.12]{Roald} and \cite[Definition 8.3]{ShulmanCruttwell}, where the notions are defined for virtual equipments. We recall the corresponding definitions.

\begin{definition}\label{def:regularframed}
    Let $C$ be a framed bicategory. We will say that a vertical morphism $f\in C_0$ in $C$, is \textbf{fully faithful} if the horizontal identity square
    \begin{center}

\tikzset{every picture/.style={line width=0.75pt}} 

\begin{tikzpicture}[x=0.75pt,y=0.75pt,yscale=-1,xscale=1]

\draw [color={rgb, 255:red, 208; green, 2; blue, 27 }  ,draw opacity=1 ]   (281,180.6) -- (331,180.6) ;
\draw [color={rgb, 255:red, 208; green, 2; blue, 27 }  ,draw opacity=1 ]   (281,230.6) -- (331,230.6) ;
\draw [color={rgb, 255:red, 0; green, 0; blue, 0 }  ,draw opacity=1 ]   (281,180.6) -- (281,230.6) ;
\draw [color={rgb, 255:red, 0; green, 0; blue, 0 }  ,draw opacity=1 ]   (331,180.6) -- (331,230.6) ;

\draw (267.14,201.26) node [anchor=north west][inner sep=0.75pt]  [font=\tiny]  {$f$};
\draw (337.71,201.26) node [anchor=north west][inner sep=0.75pt]  [font=\tiny]  {$f$};
\draw (300.29,202.11) node [anchor=north west][inner sep=0.75pt]  [font=\tiny]  {$U$};

\end{tikzpicture}

    \end{center}
    \noindent is cartesian. If the horizontal identity square above is opcartesian, we will say that $f$ is \textbf{absolutely dense}. We will say that the framed bicategory $C$ is \textbf{fully faithful/absolutely dense}, if every morphism $f$ in $C_0$ is fully faithful/absolutely dense. 
\end{definition}
\noindent In order to keep the exposition of the paper short, we will mostly treat fully faithful framed bicategories, but the arguments and results for absolutely dense framed bicategories are analogous. The interest on fully faithful framed bicategories mainly stems from the following observation. Compare with the observations made in Section \ref{ss:framed}.

\begin{obs}\label{obs:SquaresRegFramed}
\noindent In a fully faithful framed bicategory, every square of the form
\begin{center}

\tikzset{every picture/.style={line width=0.75pt}} 

\begin{tikzpicture}[x=0.75pt,y=0.75pt,yscale=-1,xscale=1]

\draw [color={rgb, 255:red, 208; green, 2; blue, 27 }  ,draw opacity=1 ]   (281,199.6) -- (331,199.6) ;
\draw [color={rgb, 255:red, 0; green, 0; blue, 0 }  ,draw opacity=1 ]   (281,149.6) -- (281,199.6) ;
\draw [color={rgb, 255:red, 0; green, 0; blue, 0 }  ,draw opacity=1 ]   (331,149.6) -- (331,199.6) ;
\draw [color={rgb, 255:red, 0; green, 0; blue, 0 }  ,draw opacity=1 ]   (281,149.6) -- (331,149.6) ;

\draw (261.6,167) node [anchor=north west][inner sep=0.75pt]  [font=\scriptsize]  {$hf$};
\draw (337.6,167.2) node [anchor=north west][inner sep=0.75pt]  [font=\scriptsize]  {$kf$};
\draw (301.47,132.73) node [anchor=north west][inner sep=0.75pt]  [font=\scriptsize]  {$\beta $};
\draw (298.61,169.3) node [anchor=north west][inner sep=0.75pt]  [font=\scriptsize]  {$\Phi $};

\end{tikzpicture}

\end{center}
\noindent can be factored uniquely as
\begin{center}

\tikzset{every picture/.style={line width=0.75pt}} 

\begin{tikzpicture}[x=0.75pt,y=0.75pt,yscale=-1,xscale=1]

\draw [color={rgb, 255:red, 208; green, 2; blue, 27 }  ,draw opacity=1 ]   (301,219.6) -- (351,219.6) ;
\draw [color={rgb, 255:red, 0; green, 0; blue, 0 }  ,draw opacity=1 ]   (301,169.6) -- (301,219.6) ;
\draw [color={rgb, 255:red, 0; green, 0; blue, 0 }  ,draw opacity=1 ]   (351,169.6) -- (351,219.6) ;
\draw [color={rgb, 255:red, 208; green, 2; blue, 27 }  ,draw opacity=1 ]   (301,169.6) -- (351,169.6) ;
\draw [color={rgb, 255:red, 0; green, 0; blue, 0 }  ,draw opacity=1 ]   (301.02,119.65) -- (301.02,169.65) ;
\draw [color={rgb, 255:red, 0; green, 0; blue, 0 }  ,draw opacity=1 ]   (351.02,119.65) -- (351.02,169.65) ;
\draw [color={rgb, 255:red, 0; green, 0; blue, 0 }  ,draw opacity=1 ]   (301.02,119.65) -- (351.02,119.65) ;

\draw (285.14,187.31) node [anchor=north west][inner sep=0.75pt]  [font=\scriptsize]  {$f$};
\draw (357.6,187.2) node [anchor=north west][inner sep=0.75pt]  [font=\scriptsize]  {$f$};
\draw (320.85,103.19) node [anchor=north west][inner sep=0.75pt]  [font=\scriptsize]  {$\beta $};
\draw (284.22,139.02) node [anchor=north west][inner sep=0.75pt]  [font=\scriptsize]  {$h$};
\draw (356.52,138.25) node [anchor=north west][inner sep=0.75pt]  [font=\scriptsize]  {$k$};
\draw (318.24,139.78) node [anchor=north west][inner sep=0.75pt]  [font=\scriptsize]  {$\Psi $};
\draw (319.24,190.78) node [anchor=north west][inner sep=0.75pt]  [font=\scriptsize]  {$U$};

\end{tikzpicture}

\end{center}
\noindent In particular, every square of the form
\begin{center}

\tikzset{every picture/.style={line width=0.75pt}} 

\begin{tikzpicture}[x=0.75pt,y=0.75pt,yscale=-1,xscale=1]

\draw [color={rgb, 255:red, 208; green, 2; blue, 27 }  ,draw opacity=1 ]   (301,219.6) -- (351,219.6) ;
\draw [color={rgb, 255:red, 0; green, 0; blue, 0 }  ,draw opacity=1 ]   (301,169.6) -- (301,219.6) ;
\draw [color={rgb, 255:red, 0; green, 0; blue, 0 }  ,draw opacity=1 ]   (351,169.6) -- (351,219.6) ;
\draw [color={rgb, 255:red, 0; green, 0; blue, 0 }  ,draw opacity=1 ]   (301,169.6) -- (351,169.6) ;

\draw (284.6,187) node [anchor=north west][inner sep=0.75pt]  [font=\scriptsize]  {$f$};
\draw (357.6,187.2) node [anchor=north west][inner sep=0.75pt]  [font=\scriptsize]  {$f$};
\draw (321.47,152.73) node [anchor=north west][inner sep=0.75pt]  [font=\scriptsize]  {$\beta $};
\draw (318.61,189.3) node [anchor=north west][inner sep=0.75pt]  [font=\scriptsize]  {$\Phi $};

\end{tikzpicture}

\end{center}
\noindent admits a unique representation as a vertical composite of the form
\begin{center}

\tikzset{every picture/.style={line width=0.75pt}} 

\begin{tikzpicture}[x=0.75pt,y=0.75pt,yscale=-1,xscale=1]

\draw [color={rgb, 255:red, 208; green, 2; blue, 27 }  ,draw opacity=1 ]   (321,239.6) -- (371,239.6) ;
\draw [color={rgb, 255:red, 0; green, 0; blue, 0 }  ,draw opacity=1 ]   (321,189.6) -- (321,239.6) ;
\draw [color={rgb, 255:red, 0; green, 0; blue, 0 }  ,draw opacity=1 ]   (371,189.6) -- (371,239.6) ;
\draw [color={rgb, 255:red, 208; green, 2; blue, 27 }  ,draw opacity=1 ]   (321,189.6) -- (371,189.6) ;
\draw [color={rgb, 255:red, 208; green, 2; blue, 27 }  ,draw opacity=1 ]   (321.02,139.65) -- (321.02,189.65) ;
\draw [color={rgb, 255:red, 208; green, 2; blue, 27 }  ,draw opacity=1 ]   (371.02,139.65) -- (371.02,189.65) ;
\draw [color={rgb, 255:red, 0; green, 0; blue, 0 }  ,draw opacity=1 ]   (321.02,139.65) -- (371.02,139.65) ;

\draw (305.14,207.31) node [anchor=north west][inner sep=0.75pt]  [font=\scriptsize]  {$f$};
\draw (377.6,207.2) node [anchor=north west][inner sep=0.75pt]  [font=\scriptsize]  {$f$};
\draw (340.85,123.19) node [anchor=north west][inner sep=0.75pt]  [font=\scriptsize]  {$\beta $};
\draw (338.24,159.78) node [anchor=north west][inner sep=0.75pt]  [font=\scriptsize]  {$\Psi $};
\draw (339.24,210.78) node [anchor=north west][inner sep=0.75pt]  [font=\scriptsize]  {$U$};

\end{tikzpicture}

\end{center}
\end{obs}

\begin{ex}\label{ex:fullyfaithful}
In the framed bicategory $\Prof$ of categories, functors, profunctors, and natural transformations a vertical arrow $F:C\to D$ is fully faithful/absolutely dense, precisely when $F$ is a fully faithful/absolutely dense functor. In the framed bicategory $\Span$ of spans over $\Set$, a vertical morphism $f:X\to Y$ is fully faithful/absolutely dense when $f$ is injective/surjective. Likewise, in the framed bicategory $\Rel$ of relations, a vertical morphism $f$ is fully faithful/absolutely dense when $f$ is injective/surjective. The framed bicategories $\Prof, \Span, \Rel$ are thus not fully faithful nor absolutely dense, nor is the framed bicategory of bimodules $\Mod (K)$ of a framed bicategory $K$ in general, see \cite[Section 8]{ShulmanCruttwell}.   
\end{ex}
\noindent Although Example \ref{ex:fullyfaithful} might be discouraging in the search for example of fully faithful/absolutely dense framed bicategories, the next example says that every framed bicategory naturally contains a fully faithful and absolutely dense framed sub-bicategory. 

\begin{obs}\label{obs:regularinvertible}
    Let $C$ be a framed bicategory. Every vertical isomorphism $f$ in $C$ is fully faithful and absolutely dense, see \cite[Example 4.13]{Roald}. In particular, if $C_0$ is a groupoid, then $C$ is fully faithful and absolutely dense. Given a framed bicategory $C$, $C^\ast$ denotes the sub-double category of $C$ whose squares are
\begin{center}

\tikzset{every picture/.style={line width=0.75pt}} 

\begin{tikzpicture}[x=0.75pt,y=0.75pt,yscale=-1,xscale=1]

\draw [color={rgb, 255:red, 0; green, 0; blue, 0 }  ,draw opacity=1 ]   (301,200.6) -- (351,200.6) ;
\draw [color={rgb, 255:red, 0; green, 0; blue, 0 }  ,draw opacity=1 ]   (301,250.6) -- (351,250.6) ;
\draw [color={rgb, 255:red, 0; green, 0; blue, 0 }  ,draw opacity=1 ]   (301,200.6) -- (301,250.6) ;
\draw [color={rgb, 255:red, 0; green, 0; blue, 0 }  ,draw opacity=1 ]   (351,200.6) -- (351,250.6) ;

\draw (288.14,221.26) node [anchor=north west][inner sep=0.75pt]  [font=\tiny]  {$f$};
\draw (357.71,221.26) node [anchor=north west][inner sep=0.75pt]  [font=\tiny]  {$g$};
\draw (320.29,222.11) node [anchor=north west][inner sep=0.75pt]  [font=\tiny]  {$\varphi$};
\draw (322.21,257.26) node [anchor=north west][inner sep=0.75pt]  [font=\tiny]  {$\beta $};
\draw (319.96,188.76) node [anchor=north west][inner sep=0.75pt]  [font=\tiny]  {$\alpha $};

\end{tikzpicture}

\end{center}
\noindent where the left and right frames $f,g$ are vertical isomorphisms. Thus defined $C^\ast$ is a fully faithful and absolutely dense sub-framed bicategory of $C$. The framed bicategories $\Prof^\ast, \Span^\ast, \Rel^\ast$ are thus fully faithful and absolutely continuous. The framed bicategory $\Bord$ of bordisms is thus also fully faithful and absolutely continuous.
\end{obs}

\noindent The following observation does better in embedding a fully faithful/absolutely dense framed bicategory canonically in an framed bicategory $C$.

\begin{obs}\label{obs:fullyfaithful}
    Let $C$ be a framed bicategory. The vertical composition of fully faithful vertical morphisms is again fully faithful. If we write $\tilde{C}$ for the sub-double category of $C$ formed by squares of the form
    \begin{center}

\tikzset{every picture/.style={line width=0.75pt}} 

\begin{tikzpicture}[x=0.75pt,y=0.75pt,yscale=-1,xscale=1]

\draw [color={rgb, 255:red, 0; green, 0; blue, 0 }  ,draw opacity=1 ]   (260.33,128.63) -- (310.33,128.63) ;
\draw [color={rgb, 255:red, 0; green, 0; blue, 0 }  ,draw opacity=1 ]   (260.33,128.63) -- (260.33,178.63) ;
\draw [color={rgb, 255:red, 0; green, 0; blue, 0 }  ,draw opacity=1 ]   (310.33,128.63) -- (310.33,178.63) ;
\draw [color={rgb, 255:red, 0; green, 0; blue, 0 }  ,draw opacity=1 ]   (260.33,178.63) -- (310.33,178.63) ;

\draw (281.03,112.73) node [anchor=north west][inner sep=0.75pt]  [font=\scriptsize]  {$\alpha $};
\draw (315.69,146.43) node [anchor=north west][inner sep=0.75pt]  [font=\scriptsize]  {$g$};
\draw (247.09,146.03) node [anchor=north west][inner sep=0.75pt]  [font=\scriptsize]  {$f$};
\draw (280.63,183.13) node [anchor=north west][inner sep=0.75pt]  [font=\scriptsize]  {$\beta $};
\draw (280.23,146.73) node [anchor=north west][inner sep=0.75pt]  [font=\scriptsize]  {$\varphi $};

\end{tikzpicture}

    \end{center}
    where $f,g$ are fully faithful, then $\tilde{C}$ is a fully faithful framed bicategory containing $C^\ast$. Similarly, the sub-double category $\hat{C}$ of $C$ formed by squares of the form
    \begin{center}

\tikzset{every picture/.style={line width=0.75pt}} 

\begin{tikzpicture}[x=0.75pt,y=0.75pt,yscale=-1,xscale=1]

\draw [color={rgb, 255:red, 0; green, 0; blue, 0 }  ,draw opacity=1 ]   (260.33,128.63) -- (310.33,128.63) ;
\draw [color={rgb, 255:red, 0; green, 0; blue, 0 }  ,draw opacity=1 ]   (260.33,128.63) -- (260.33,178.63) ;
\draw [color={rgb, 255:red, 0; green, 0; blue, 0 }  ,draw opacity=1 ]   (310.33,128.63) -- (310.33,178.63) ;
\draw [color={rgb, 255:red, 0; green, 0; blue, 0 }  ,draw opacity=1 ]   (260.33,178.63) -- (310.33,178.63) ;

\draw (281.03,112.73) node [anchor=north west][inner sep=0.75pt]  [font=\scriptsize]  {$\alpha $};
\draw (315.69,146.43) node [anchor=north west][inner sep=0.75pt]  [font=\scriptsize]  {$g$};
\draw (247.09,146.03) node [anchor=north west][inner sep=0.75pt]  [font=\scriptsize]  {$f$};
\draw (280.63,183.13) node [anchor=north west][inner sep=0.75pt]  [font=\scriptsize]  {$\beta $};
\draw (280.23,146.73) node [anchor=north west][inner sep=0.75pt]  [font=\scriptsize]  {$\varphi $};

\end{tikzpicture}

    \end{center}
    where $f,g$ are now absolutely dense, is an absolutely dense framed bicategory, also containing $C^\ast$. 
\end{obs}

\begin{ex}
    The framed bicategories $\tilde{\Prof}$ and $\hat{\Prof}$ are fully faithful and absolutely dense sub-double categories of $\Prof$, both properly containing $\Prof^\ast$. Similarly $\tilde{\Span}$ and $\hat{\Span}$ are proper fully faithful and absolutely dense sub-framed bicategories of $\Span$, properly containing $\Span^\ast$. Similarly for $\Rel$.
\end{ex}

\noindent The following proposition says that the framed bicategories $\tilde{C}$ and $\hat{C}$ defined in Observation \ref{obs:fullyfaithful} satisfy a universal property with respect to strong framed functors, c.f. \cite[Definition 6.14]{ShulmanDerived}, expressing the fact that they are maximal among fully faithful/absolutely dense sub-framed bicategories of the framed bicategory $C$. 

\begin{prop}\label{prop:tildehatuniversal}
    Let $C,D$ be framed bicategories. Let $F:C\to D$ be a strong framed functor. Suppose $D$ is fully faithful. There exists a unique factorization of $F$ as
    \begin{center}
    \begin{tikzpicture}
\matrix(m)[matrix of math nodes, row sep=2em, column sep=2em,text height=1.5ex, text depth=0.25ex]
{D&&C\\
&\tilde{C}&\\};
\path[->,font=\scriptsize,>=angle 90]
(m-1-1) edge node [above]{$F$} (m-1-3)
        edge [dotted] node [left]{$\tilde{F}$} (m-2-2)
(m-2-2) edge node [right]{} (m-1-3)

;
\end{tikzpicture}
\end{center}
along a strong framed functor $\tilde{F}$, where the arrow on the right is the inclusion strict framed functor $\tilde{C}\to C$. In particular a sub-framed bicategory $D$ of a framed bicategory $C$ is fully faithful if and only if it is a sub-framed bicategory of $\tilde{C}$. Likewise, if we assume that $D$ is absolutely dense, then there exists a unique factorization of $F$ as
\begin{center}
    \begin{tikzpicture}
\matrix(m)[matrix of math nodes, row sep=2em, column sep=2em,text height=1.5ex, text depth=0.25ex]
{D&&C\\
&\hat{C}&\\};
\path[->,font=\scriptsize,>=angle 90]
(m-1-1) edge node [above]{$F$} (m-1-3)
        edge [dotted] node [left]{$\hat{F}$} (m-2-2)
(m-2-2) edge node [right]{} (m-1-3)

;
\end{tikzpicture}
\end{center}
\noindent where $\hat{F}$ is again a strong framed functor and the arrow on the right is the inclusion strict framed functor $\hat{C}\to C$, and a sub-framed bicategory $D$ of $C$ is absolutely dense if and only if $D$ is a sub-framed bicategory of $\hat{C}$.
\end{prop}

\begin{proof}
    Let $F:D\to C$ be a strong framed functor. The square $U_{F(f)}$ is cartesian in $C$ for every vertical morphism $f$ in $D$. This follows from \cite[Proposition 6.4]{ShulmanFramed}, so that the square $F(U_f)$ is Cartesian in $C$, and thus $U_{F(f)}$ is cartesian. From this it follows the image $F(\varphi)$ of every square $\varphi$ in $D$ is a square in the fully faithful sub-framed bicategory $\tilde{C}$ of $C$. The first part of the statement of the proposition follows from this. An analogous argumen proves the second part.
\end{proof}
\begin{ex}\label{ex:Eq}
    The sub-double category $\Eq$ of $\Prof$ formed by squares having equivalences as left and right frames, and representable functors as horizontal arrows, is thus an example of a fully faithul sub-framed bicategory of $\Prof$ that is not absolutely dense.
\end{ex}

\noindent We can restate the universal properties stated in Proposition \ref{prop:tildehatuniversal} in terms of a pair of adjunctions. Write $\mathcal{F}r\mathcal{B}i$ for the 2-category of framed bicategories, strong framed functors, and framed double natural transformations of \cite[Proposition 6.17]{ShulmanFramed}. Write $\mathcal{F}r\mathcal{B}i^{ff}$ and $\mathcal{F}r\mathcal{B}i^{ad}$ for the full sub 2-categories of $\mathcal{F}r\mathcal{B}i$ generated by fully faithful framed bicategories and by absolutely dense framed bicategories respectively. By Proposition \ref{prop:tildehatuniversal} the constructions $C\mapsto \tilde{C}$ and $C\mapsto \hat{C}$ extend to strict 2-functors $\tilde{\_}$ and $\hat{\_}$ from $\mathcal{F}r\mathcal{B}i$ to $\mathcal{F}r\mathcal{B}i^{ff}$ and $\mathcal{F}r\mathcal{B}i^{ad}$ respectively. 

\begin{prop}\label{prop:adjunctions}
The 2-functor functor $\tilde{\_}$ fits into an adjuction
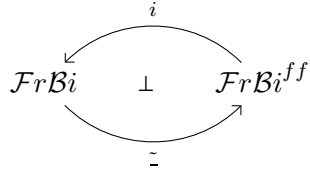
       \begin{center}
    \begin{tikzpicture}
\matrix(m)[matrix of math nodes, row sep=2em, column sep=2em,text height=1.5ex, text depth=0.25ex]
{\mathcal{F}r\mathcal{B}i&&\mathcal{F}r\mathcal{B}i^{ff}\\};
\path[->,font=\scriptsize,>=angle 90]
(m-1-1) edge [bend right=45] node [below]{$\tilde{\_}$} (m-1-3)
(m-1-3) edge [bend right=45] node [above]{$i$} (m-1-1)
(m-1-1) edge [white] node [black]{$\perp$} (m-1-3)

;
\end{tikzpicture}
\end{center}
\noindent where $i$ is the inclusion 2-functor of $\mathcal{F}r\mathcal{B}i^{ff}$ in $ \mathcal{F}r\mathcal{B}i$. Likewise, the 2-functor $\hat{\_}$ fits into an adjuction
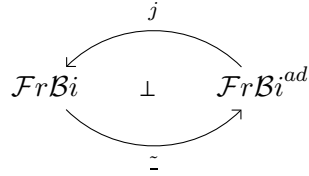
   \begin{center}
    \begin{tikzpicture}
\matrix(m)[matrix of math nodes, row sep=2em, column sep=2em,text height=1.5ex, text depth=0.25ex]
{\mathcal{F}r\mathcal{B}i&&\mathcal{F}r\mathcal{B}i^{ad}\\};
\path[->,font=\scriptsize,>=angle 90]
(m-1-1) edge [bend right=45] node [below]{$\tilde{\_}$} (m-1-3)
(m-1-3) edge [bend right=45] node [above]{$j$} (m-1-1)
(m-1-1) edge [white] node [black]{$\perp$} (m-1-3)

;
\end{tikzpicture}
\end{center}
\noindent where now $j$ is the inclusion 2-functor of $\mathcal{F}r\mathcal{B}i^{ad}$ in $\mathcal{F}r\mathcal{B}i$.
\end{prop}

\section{Fully faithful framed bicategories are of length 1}\label{s:regframedidexngs}
\noindent In this final section we leverage Theorem \ref{thm:initial} and Corollary \ref{cor:mainthm} to describe the structure of fully faithful and absolutely dense framed bicategories. In particular we prove that both fully faithful and abolutely dense framed bicategories are of length 1. The main piece of structure we need to do this is provided by the following proposition.
\begin{prop}\label{lem:main}
 Let $C$ be a fully faithful framed bicategory. There exists a $\pi_2$-opindexing $\Phi$ on $H^\ast C$ such that $C$ induces $\Phi$. Likewise if $C$ is absolutely dense, then there exists a $\pi_2$-indexing on $H^\ast C$ such that $C$ induces $\Phi$. 
 \end{prop}
 \begin{proof}
 We prove that every fully faithful framed bicategory induces a $\pi_2$-opindexing. Assume $C$ is fully faithful. Let $f:a\to b$ be a vertical morphism in $C$. Let $\varphi\in\pi_2(C,b)$. Write $\Phi_f(\varphi)$ for the unique square in $\pi_2(C,a)$ appearing in the unique factorization of 
 \begin{center}

\tikzset{every picture/.style={line width=0.75pt}} 



 \end{center}
 as in \ref{obs:SquaresRegFramed}. We prove that the assignment $a\mapsto \pi_2(C,a)$ and $f\mapsto \Phi_f$ forms a $\pi_2$-opindexing on $H^\ast C$. We first prove that $\Phi_f$ is a morphism of monoids from $\pi_2(C,b)$ to $\pi_2(C,a)$. Let $\varphi,\varphi'\in\pi_2(C,b)$. Consider the composite square
    \begin{center}

\tikzset{every picture/.style={line width=0.75pt}} 

%

    \end{center}
Consider the unique factorization of the upper square as
    \begin{center}

\tikzset{every picture/.style={line width=0.75pt}} 

%

    \end{center}
Substituting, we obtain the equality
    \begin{center}

\tikzset{every picture/.style={line width=0.75pt}} 

%

    \end{center}
Now consider the unique factorization
    \begin{center}

\tikzset{every picture/.style={line width=0.75pt}} 

%

    \end{center}
substuting above, we obtain
    \begin{center}

\tikzset{every picture/.style={line width=0.75pt}} 

%

    \end{center}
The unique factorization of $\varphi'\boxminus\varphi$ along $U_f$ is thus the composite square appearing in the right hand side of the equality above. We conclude that $\Phi_f(\varphi\boxminus\varphi')=\Phi_f(\varphi')\boxminus\Phi_f(\varphi)$. It is easily seen that $\Phi_f(id_b)=id_a$ and thus $\Phi_f:\pi_2(C,b)\to\pi_2(C,a)$ is a monoid morphism for every $f:a\to b$. The fact that the assignmnet $f\mapsto \Phi_f$ is a functor follows from our running assumption that $L\times R$ is split and normal, see \ref{ss:framed}.
\end{proof}

\noindent For the rest of this section we will denote by $\Phi$ the $\pi_2$-opindexing induced by a fully faithful framed bicategory $C$ as in Proposition \ref{lem:main}. As a direct corollary of Proposition \ref{lem:main}, Theorem \ref{thm:initial} and Corollary \ref{cor:mainthm} we obtain the following theorem, which says that any fully faithful/absolutely dense is, in a sense, parametrized by a double category of the form $\crossb$. We will obtain in particular that every nonglobular square in a fully faithful/absoutely dense framed bicategory is canonical and thus fully faithful/absolutely dense framed bicategories are of length 1.
\begin{thm}\label{thm:mainfullyfaithful}
    Let $C$ be a framed bicategory. If $C$ is fully faithful or absolutely dense, then $\ell C=1$. Moreover, if $C$ is fully faithful/absolutely dense and $\Phi$ is the $\pi_2$-opindexing/$\pi_2$-indexing induced by $C$ as in Proposition \ref{lem:main}, then there exists a unique strict double functor
    \[!:H C\rtimes_\Phi C_0\to C\]
    such that $H^\ast !=id_{H^\ast C}$. Moreover, $!$ is full on globularly generated squares of $C$.
\end{thm}
\noindent We end this section with a few examples of $\pi_2$-opindexings induced by some of the fully faithful framed bicategories presented in Section \ref{s:fully faithful}.
\begin{ex}\label{ex:rho2}
    Let $X$ be a topological space. Let $\rho^\square_2(X)$ be homotopy double groupoid with connections of $X$, see \cite{BrownMosa}. $\rho^\square_2(X)$ is a fully faithful and absolutely dense framed bicategory. The category of objects of $\rho^\square_2(X)$ is the fundamental groupoid $\Pi_1(X)$ of $X$. Let $\Phi:\Pi_1(X)\to\commMon$ be the $\pi_2$-indexing associated to $\rho^\square_2(X)$. Given an object $x$ of $\rho^\square_2(X)$, the commutative monoid $\pi_2(\rho^\square_2(X),x)$ is the second fundamental group $\pi_2(X,x)$. Let $\gamma:x\to y$ be a homotopy equvalence of paths from $x$ to $y$ in $X$. The morphism $\Phi_\gamma:\pi_2(X,y)\to\pi_2(X,x)$ in Lemma \ref{lem:main} is the morphism $a\mapsto \gamma a\gamma^{-1}$, $a\in \pi_2(X,y)$. Similar computations for the $\pi_2$-opindexing associated to $\Bord,\Rel ^\ast,\Span ^\ast$, etc.
\end{ex}
\begin{ex}\label{ex:eq}
    We determine the $\pi_2$-opindexing associated to the fully faithful framed bicategory $\Eq$ appearing in Example \ref{ex:Eq}. Let $C$ be a category. The commutative monoid $\pi_2(\Eq,C)$ is the commutative monoid of natural endomorphisms $\mbox{\textbf{Nat}}(C)$ of $id_C$. The $\pi_2$-opindexing $\Phi$ in Lemma \ref{lem:main} thus associates to every category $C$, the monoid $\mbox{\textbf{Nat}}(C)$. Given an equivalence $F:C\to D$, the morphism $\Phi_F:\mbox{\textbf{Nat}}(D)\to\mbox{\textbf{Nat}}(C)$ associates to every $\nu:id_D\Rightarrow id_D$, the natural transformation provided by the pasting diagram
    \begin{center}
    \begin{tikzpicture}
\matrix(m)[matrix of math nodes, row sep=2em, column sep=2em,text height=1.5ex, text depth=0.25ex]
{C&D&D&C\\};
\path[->,font=\scriptsize,>=angle 90]
(m-1-1) edge node [below]{$F$} (m-1-2)
(m-1-1) edge [red,bend left=65]node []{} (m-1-4)
(m-1-1) edge [white,bend left=35]node [black]{$\Downarrow\eta^{-1}$} (m-1-4)
(m-1-1) edge [red,bend right=65]node []{} (m-1-4)
(m-1-1) edge [white,bend right=35]node [black]{$\Downarrow\eta$} (m-1-4)
(m-1-2) edge [red, bend right=45]node [below]{} (m-1-3)
(m-1-2) edge [red, bend left=45]node [below]{} (m-1-3)
(m-1-2) edge [white]node [black]{$\Downarrow\nu$} (m-1-3)
(m-1-3) edge node [below]{$F^{-1}$} (m-1-4)

;
\end{tikzpicture}
\end{center}
    where $F^{-1},\eta$ are part of an adjoint equivalence $(F,F^{-1},\eta,\epsilon)$.
\end{ex}

\noindent A natural question is when the double functor $!$ appearing in Theorems \ref{thm:initial} and \ref{thm:mainfullyfaithful} is a double isomorphism onto $\gamma C$, i.e. when $!$ is injective on squares, and when it is not. The following is an example of when it is not. The argument rests on Observation \ref{obs:squarescrosspsame}.

\begin{ex}\label{ex:nonisomorphism}
The double functor appearing in the statement of Theorems \ref{thm:initial} and $\ref{thm:mainfullyfaithful}$ is, in general, not a double isomorphism. To see this, consider the double category $\Span ^\ast$. By Observation \ref{obs:regularinvertible}, $\Span ^\ast$ is a fully faithful and absolutely dense framed bicategory. Consider diagrams of the form
\begin{center}
    \begin{tikzpicture}
\matrix(m)[matrix of math nodes, row sep=2em, column sep=2em,text height=1.5ex, text depth=0.25ex]
{X&X&X&&X&X&X\\
X&X&X&&X&X&X\\
X&X&X&&X&X&X\\
X&X&X&&X&X&X\\};
\path[->,font=\scriptsize,>=angle 90]
(m-1-2) edge node [above]{$\alpha$} (m-1-1)
        edge node [above]{$\alpha$} (m-1-3)
        edge node [left]{$\alpha$} (m-2-2)
(m-1-1) edge [red]node {}(m-2-1)
(m-1-3) edge [red]node {}(m-2-3)
(m-2-2) edge [red]node [above]{} (m-2-1)
        edge [red]node [left]{} (m-2-3)
        edge node [left]{$f$} (m-3-2)
(m-2-1) edge node [left]{$f$} (m-3-1)  
(m-2-3) edge node [right]{$f$} (m-3-3)  
(m-3-2) edge [red]node [below]{} (m-3-1)
        edge [red]node [below]{} (m-3-3)
        edge node [left]{$\beta$} (m-4-2)
(m-3-1) edge [red]node {} (m-4-1)
(m-3-3) edge [red]node {} (m-4-3)
(m-4-2) edge node [below]{$\gamma$} (m-4-1)  
        edge node [below]{$\gamma$} (m-4-3)

(m-1-6) edge node [above]{$\alpha$} (m-1-5)
        edge node [above]{$\alpha$} (m-1-7)
        edge node [left]{$\alpha$} (m-2-6)
(m-1-5) edge [red]node []{}(m-2-5)
(m-1-7) edge [red]node []{}(m-2-7)
(m-2-6) edge [red]node {} (m-2-5)
        edge [red]node [right]{} (m-2-7)
        edge node [left]{$f$} (m-3-6)
(m-2-5) edge node [left]{$f$} (m-3-5)  
(m-2-7) edge node [right]{$f$} (m-3-7)  
(m-3-6) edge [red]node {} (m-3-5)
        edge [red]node {} (m-3-7)
(m-3-5) edge[red] node[left]{} (m-4-5)
(m-3-6) edge node[left]{$\beta'$} (m-4-6)
(m-3-7) edge[red] node[left]{} (m-4-7)
(m-4-6) edge node[below]{$\gamma$} (m-4-5)
        edge node[below]{$\gamma$} (m-4-7)

;
\end{tikzpicture}
\end{center}
\noindent where $X$ is an infinite set, $f;X\to X$ is a bijection such that $f\neq id_X$, $\gamma:X\to X$ is a surjective function that is not bijective, $\beta,\beta'$ are right inverses of $\gamma$ such that $\beta\neq\beta'$,  and such that $\beta f\alpha=\beta' f\alpha$, then if we write $\overline{\alpha}$ for the top diagram in the above two squares, and $\underline{\beta},\underline{\beta'}$ for the bottom diagrams above, then the squares $\psi=(\underline{\beta}, f, \overline{\alpha})$ and $\psi'=(\underline{\beta'}, f, \overline{\alpha})$ are different in $H\Span ^\ast\rtimes_\Phi \Set^\ast$, but the images $!(\psi)$ and $!(\psi')$ under the double functor $!$ of Theorem \ref{thm:mainfullyfaithful} are both equal to the square

\begin{center}
    \begin{tikzpicture}
\matrix(m)[matrix of math nodes, row sep=2em, column sep=2em,text height=1.5ex, text depth=0.25ex]
{X&X&X\\
X&X&X\\};
\path[->,font=\scriptsize,>=angle 90]
(m-1-2) edge node [above]{$\alpha$} (m-1-1)
        edge node [above]{$\alpha$} (m-1-3)
        edge node [left]{$\eta$} (m-2-2)
(m-1-1) edge node [left]{$f$}(m-2-1)
(m-1-3) edge node [right]{$f$}(m-2-3)
(m-2-2) edge node [below]{$\gamma$} (m-2-1)
        edge node [below]{$\gamma$} (m-2-3)

;
\end{tikzpicture}
\end{center}
in $\Span ^\ast$, where $\beta f\alpha=\beta' f\alpha$. We conclude that in this case $!$ is not bijective on squares.
    
\end{ex}

\noindent The following Observation provides conditions for the double functor $q$ to be a bijection.
\begin{obs}
    Let $C$ be a fully faithful framed bicategory. If $C$ is a double groupoid, then the double functor $!$ appearing in Theorem \ref{thm:mainfullyfaithful} is a double isomorphism. This follows directly from the fact that for nonglobular squares of the form
    \begin{center}
        \begin{center}

\tikzset{every picture/.style={line width=0.75pt}} 

\begin{tikzpicture}[x=0.75pt,y=0.75pt,yscale=-1,xscale=1]

\draw [color={rgb, 255:red, 0; green, 0; blue, 0 }  ,draw opacity=1 ]   (300,69.4) -- (350,69.4) ;
\draw [color={rgb, 255:red, 208; green, 2; blue, 27 }  ,draw opacity=1 ]   (300,119.4) -- (350,119.4) ;
\draw [color={rgb, 255:red, 208; green, 2; blue, 27 }  ,draw opacity=1 ]   (300,69.4) -- (300,119.4) ;
\draw [color={rgb, 255:red, 208; green, 2; blue, 27 }  ,draw opacity=1 ]   (350,69.4) -- (350,119.4) ;
\draw [color={rgb, 255:red, 0; green, 0; blue, 0 }  ,draw opacity=1 ]   (300,119.4) -- (300,169.4) ;
\draw [color={rgb, 255:red, 0; green, 0; blue, 0 }  ,draw opacity=1 ]   (350,119.4) -- (350,169.4) ;
\draw [color={rgb, 255:red, 208; green, 2; blue, 27 }  ,draw opacity=1 ]   (300,169.4) -- (350,169.4) ;
\draw [color={rgb, 255:red, 0; green, 0; blue, 0 }  ,draw opacity=1 ]   (300,219) -- (350,219) ;
\draw [color={rgb, 255:red, 208; green, 2; blue, 27 }  ,draw opacity=1 ]   (300,169) -- (300,219) ;
\draw [color={rgb, 255:red, 208; green, 2; blue, 27 }  ,draw opacity=1 ]   (350,169) -- (350,219) ;

\draw (352.36,135.84) node [anchor=north west][inner sep=0.75pt]  [font=\scriptsize]  {$f$};
\draw (289.36,134.84) node [anchor=north west][inner sep=0.75pt]  [font=\scriptsize]  {$f$};
\draw (318.36,185.44) node [anchor=north west][inner sep=0.75pt]  [font=\scriptsize]  {$\varphi _{\downarrow }$};
\draw (318.21,137.5) node [anchor=north west][inner sep=0.75pt]  [font=\scriptsize]  {$U$};
\draw (317.96,84.64) node [anchor=north west][inner sep=0.75pt]  [font=\scriptsize]  {$\varphi _{\uparrow }$};

\end{tikzpicture}

\end{center}
    \end{center}
    in a double groupoid $C$ admit a unique such decomposition.
\end{obs}

\bibliographystyle{alpha}
\bibliography{biblio}

\begin{thebibliography}{OMH24}

\bibitem[BDH15]{Bartels1}
A.~Bartels, C.L. Douglas, and A.~Henriques.
\newblock Conformal nets i: Coordinate free nets.
\newblock {\em Int. Math. Res. Not}, 13:4975--5052, 2015.

\bibitem[BHS11]{BrownBook}
R.~Brown, P.J. Higgins, and R.~Sivera.
\newblock {\em Nonabelian Algebraic Topology}.
\newblock EMS Tracts in Mathematics, 2011.

\bibitem[BM99]{BrownMosa}
R.~Brown and G.~H. Mosa.
\newblock Double categories, 2-categories, thin structures and connections.
\newblock {\em Theory Appl. Categ.}, 5(7):163--1757, 1999.

\bibitem[CS10]{ShulmanCruttwell}
G.S.H. Cruttwell and M.~A. Shulman.
\newblock A unified framework for generalized multicategories.
\newblock {\em Theory Appl. Categ.}, 24(21):580--655, 2010.

\bibitem[Kou20]{Roald}
E.~R. Koudenburg.
\newblock Augmented virtual double categories.
\newblock {\em Theory Appl. Categ.}, 35(10):261--325, 2020.

\bibitem[KS74]{KellyStreet}
M.~Kelly and R.~Street.
\newblock Review of the elements of 2-categories.
\newblock {\em Lecture notes in Mathematics}, 420, 1974.

\bibitem[OMH24]{OrendainMaldonado}
J.~Orendain and R.~Maldonado-Herrera.
\newblock Internalizations of decorated bicategories via $\pi_2$-indexings.
\newblock {\em To appear in Applied Categorical Structures}, 2024.

\bibitem[Ore19a]{Orendain3}
J.~Orendain.
\newblock Free globularly generated double categories ii: The canonical double projection.
\newblock {\em Theory and Applications of Categories}, 34(42):1343--1385, 2019.

\bibitem[Ore19b]{Orendain1}
J.~Orendain.
\newblock Lifting bicategories into double categories: The globularily generated condition.
\newblock {\em Theory and Applications of Categories}, 34:80--108, 2019.

\bibitem[Ore21]{Orendain2}
J.~Orendain.
\newblock Free globularly generated double categories.
\newblock {\em Cahiers de topologie et géometrie différentielle catégoriques}, LXII Issue 3:243--302, 2021.

\bibitem[Shu08]{ShulmanFramed}
M.~A. Shulman.
\newblock Framed bicategories and monoidal fibrations.
\newblock {\em Theory and Applications Categories}, 18:650--738, 2008.

\bibitem[Shu11]{ShulmanDerived}
M.~A. Shulman.
\newblock Comparing composites of left and right derived functors.
\newblock {\em New York Journal of Mathematics}, 17:75--125, 2011.

\bibitem[Woo82]{Wood1}
R.~J. Wood.
\newblock Abstract proarrows i.
\newblock {\em Cahiers de topologie et géométrie différentielle}, 23(3):279--290, 1982.

\bibitem[Woo85]{Wood2}
R.~J. Wood.
\newblock Proarrows ii.
\newblock {\em Cahiers de topologie et géométrie différentielle}, 26(2):135--168, 1985.

\end{thebibliography}
\end{document}